\documentclass[reqno,10pt]{amsart}
\usepackage{amscd,amssymb,verbatim,array}
\usepackage{hyperref}
\usepackage{amsmath}
\usepackage[active]{srcltx} 

\numberwithin{equation}{section} \theoremstyle{plain}

\newtheorem{theorem}{Theorem}[section]
\newtheorem{lemma}[theorem]{Lemma}
\newtheorem{proposition}[theorem]{Proposition}

\newtheorem{definition}[theorem]{Definition}

\theoremstyle{definition}

\theoremstyle{remark}
\newtheorem{remark}[theorem]{Remark}

\numberwithin{equation}{section}

\newcommand{\LDet}{\operatorname{LDet}}
\newcommand{\Det}{\operatorname{Det}}
\newcommand{\RS}{\operatorname{RS}}

\newcommand{\an}{\operatorname{an}}
\newcommand{\B}{\mathcal{B}}
\newcommand{\even}{\operatorname{even}}
\newcommand{\trivial}{\operatorname{trivial}}
\newcommand{\odd}{\operatorname{odd}}

\newcommand{\Ker}{\operatorname{Ker}}
\newcommand{\Tr}{\operatorname{Tr}}

\newcommand{\im}{\operatorname{Im}}
\newcommand{\gr}{\operatorname{gr}}
\newcommand{\trs}{T^{\operatorname{RS}}}

\begin{document}

\title{Refined analytic torsion for twisted de Rham complexes}
\author{Rung-Tzung Huang}

\address{Institute of Mathematics, Academia Sinica, Nankang 11529, Taipei, Taiwan}

\email{rthuang@math.sinica.edu.tw}


\begin{abstract}
Let $E$ be a flat complex vector bundle over a closed oriented odd dimensional manifold $M$ endowed with a flat connection $\nabla$. The refined analytic torsion for $(M,E)$ was defined and studied by Braverman and Kappeler. Recently Mathai and Wu defined and studied the analytic torsion for the twisted de Rham complex with an odd degree closed differential form $H$, other than one form, as a flux and with coefficients in $E$. In this paper we generalize the construction of the refined analytic torsion to the twisted de Rham complex. We show that the refined analytic torsion of the twisted de Rham complex is independent of the choice of the Riemannian metric on $M$ and the Hermitian metric on $E$. We also show that the twisted refined analytic torsion is invariant (under a natural identification) if $H$ is deformed within its cohomology class. We prove a duality theorem, establishing a relationship between the twisted refined analytic torsion corresponding to a flat connection and its dual. We also define the twisted analogue of the Ray-Singer metric and calculate the twisted Ray-Singer metric of the twisted refined analytic torsion. In particular we show that in case that the Hermtitian connection is flat, the twisted refined analytic torsion is an element with the twisted Ray-Singer norm one.  
\end{abstract}
\maketitle

\section{Introduction}
Let $E$ be a flat complex vector bundle over a closed oriented odd dimensional manifold $M$ endowed with a flat connection $\nabla$. Braverman and Kappeler \cite{BK1,BK3,BK4,BK5,BK6} defined and studied the refined analytic torsion for $(M,E)$, which can be viewed as a refinement of the Ray-Singer torsion \cite{RS} and an analytic analogue of the Farber-Turaev torsion, \cite{FT1,FT2,Tu1,Tu2}. It was shown that the refined analytic torsion is closely related with the Farber-Turaev torsion, \cite{BK1,BK3,BK6,H}.

In \cite{MW,MW1} Mathai and Wu generalize the classical construction of the Ray-Singer torsion to the twisted de Rham complex with an odd degree closed differential form $H$, other than one form, as a flux and with coefficients in $E$. The twisted de Rham complex is the $\mathbb{Z}_2$-graded complex $(\Omega^\bullet(M,E),\nabla^H)$, where $(\Omega^\bullet(M,E)$ is the space of differential forms with coefficients in $E$ and $\nabla^H:=\nabla + H\wedge \cdot$. Its cohomology $H^{\bullet}(M,E,H)$ is called the twisted de Rham cohomology. Mathai and Wu \cite{MW} defined the analytic torsion of the twisted de Rham complex $\tau(M,E,H) \in \Det\big(H^{\bullet}(M,E,H)\big)$ as a ratio of $\zeta$-regularized determinants of partial Laplacians, multiplied by the ratio of volume elements of the cohomology groups. They showed that when $\dim M$ is odd, $\tau(M,E,H)$ is independent of the choice of the Riemannian metric on $M$ and the Hermitian metric on $E$. They also showed that the torsion $\tau(M,E,H)$ is invariant (under a natural identification) if $H$ is deformed within its cohomology class and discussed its connection with the generalized geometry \cite{Gu2}. 

In this paper we define the refined analytic torsion for the twisted de Rham complex $\rho_{\operatorname{an}}(\nabla^H)\in \Det \big( H^{\bullet} (M,E,H)\big)$. We show that the twisted refined analytic torsion $\rho_{\operatorname{an}}(\nabla^H)$ is independent of the choice of the Riemannian metric on $M$ and the Hermitian metric on $E$. We then show that the torsion $\rho_{\operatorname{an}}(\nabla^H)$ is invariant (under a natural identification) if $H$ is deformed within its cohomology class. We also establish a duality theorem, establishing a relationship between the twisted refined analytic torsion corresponding to a flat connection and its dual, which is a twisted analogue of Theorem 10.3 of \cite{BK3}. In the end we define the twisted analogue of the Ray-Singer metric and then calculate the twisted Ray-Singer norm of the twisted refined analytic torsion. In particular we show that in case of flat Hermtitian metric, the twisted refined analytic torsion is a canonical choice of an element with the twisted Ray-Singer norm one. 

The paper is organized as follows. In Section \ref{S:rtrev}, we review some standard materials about determinant lines of a $\mathbb{Z}_2$-graded finite dimensional complex. Then we define and calculate the refined torsion of the $\mathbb{Z}_2$-graded finite dimensional complex with a chirality operator. In Section \ref{S:gdtoso}, we define the graded determinant of the twisted version of the odd signature operator of a flat vector bundle $E$ over a closed oriented odd dimensional manifold $M$. We use this graded determinant to define a canonical element $\rho_H$ of the determinant line of the twisted de Rham cohomology of the vector bundle $E$. We study the relationship between this graded determinant and the $\eta$-invariant of the twisted odd signature operator. In Section \ref{S:madrat}, we first study the metric dependence of the canonical element $\rho_H$ and then use this element to construct the refined analytic torsion twisted by the flux form $H$. We show that the twisted refined analytic torsion is independent of the metric $g^M$ and the representative $H$ in the cohomology class $[H]$. In Section \ref{S:dttrat}, we first review the concept of the dual of a complex and construct a natural isomorphism between the determinant lines of a $\mathbb{Z}_2$-graded complex and its dual. We then establish a relationship between the twisted refined analytic torsion corresponding to a flat connection and that of its dual. In Section \ref{S:cwtat}, we first define the twisted Ray-Singer metric and then calculate the twisted Ray-Singer norm of the twisted refined analytic torsion.

Throughout this paper, the bar over an integer means taking the value modulo $2$.

\subsection*{Acknowledgement}
The author would like to thank Maxim Braverman for suggesting this problem.

\section{The refined torsion of a $\mathbb{Z}_2$-graded finite dimensional complex with a chirality operator}\label{S:rtrev}

In this section we first review some standard materials about determinant lines of a $\mathbb{Z}_2$-graded finite dimensional complex. Then we define and calculate the refined torsion of the $\mathbb{Z}_2$-graded finite dimensional complex with a chirality operator. The contents are $\mathbb{Z}_2$-graded analogues of Section 2, Section 4 and Section 5 of \cite{BK3}.
Throughout this section $\bf{k}$ is a field of characteristic zero.
\subsection{The determinant line of a $\mathbb{Z}_2$-graded finite dimensional complex}
Given a $\bf{k}$-vector space $V$ of dimension $n$, the determinant line of $V$ is the line $\Det(V):=\wedge^nV$, where $\wedge^nV$ denotes the $n$-th exterior power of $V$. By definition, we set $\operatorname{Det}(0):=\bf{k}$. Further, we denote by $\Det(V)^{-1}$ the dual line of $\Det(V)$.  
Let 
\begin{equation}\label{E:utdetline} 
  0 \stackrel{}{\longrightarrow} \ C^0 \ \stackrel{\partial_0}{\longrightarrow}\ C^1\ \stackrel{\partial_1}{\longrightarrow} \ \cdots \stackrel{\partial_{m-1}}{\longrightarrow}\ \ C^m \  \stackrel{\partial_m}{\longrightarrow} 0
\end{equation}
be an odd length, i.e. $m=2r-1$ being a positive odd integer, cochain complex of finite dimensional $\bf{k}$-vector spaces. Set 
$$C^{\bar{0}}=C^{\operatorname{even}}=\bigoplus_{i=0}^{r-1} C^{2i}, \qquad C^{\bar{1}}=C^{\operatorname{odd}}=\bigoplus_{i=0}^{r-1} C^{2i+1}.$$
Let 
\begin{equation}\label{E:detline} 
 (C^\bullet,d) \  : \  \cdots \stackrel{d_{\bar{1}}}{\longrightarrow} \ C^{\bar{0}} \ \stackrel{d_{\bar{0}}}{\longrightarrow}\ C^{\bar{1}}\ \stackrel{d_{\bar{1}}}{\longrightarrow} \ C^{\bar{0}} \ \stackrel{d_{\bar{0}}}{\longrightarrow} \cdots
\end{equation}
be a $\mathbb{Z}_2$-graded cochain complex of finite dimensional $\bf{k}$-vector spaces. For example, we can choose $d_{\bar{k}}=\sum_{i, i=k\, \text{mod}\, 2}\partial_i$. Denote by $H^{\bar{k}}(d_{\bar{k}}), (k=0,1)$ its cohomology. Set
\begin{equation}
\Det(C^\bullet)\,:=\,\Det \big( C^{\bar{0}} \big) \otimes \Det \big( C^{\bar{1}} \big)^{-1}, \qquad \Det(H^\bullet(d))\,:=\, \Det\big( H^{\bar{0}}(d_{\bar{0}})\big) \otimes \Det\big( H^{\bar{1}}(d_{\bar{1}})\big)^{-1}.
\end{equation}

\subsection{The fusion isomorphisms}(cf. \cite[Subsection 2.3]{BK3})
For two finite dimensional $\bf{k}$-vector spaces $V$ and $W$, we denote by $\mu_{V,W}$ the canonical fusion isomorphism,
\begin{equation}\label{E:fusio}
\mu_{V,W} : \Det(V) \otimes \Det(W) \to \Det(V \oplus W). 
\end{equation}
For $v \in \Det(V)$, $w \in \Det(W)$, we have
\begin{equation}\label{E:fusio1}
\mu_{V,W}(v \otimes w)= (-1)^{\dim V \cdot \dim W} \mu_{W,V}(w \otimes v).
\end{equation}
By a slight abuse of notation, denote by $\mu_{V,W}^{-1}$ the transpose of the inverse of $\mu_{V,W}$.

Similarly, if $V_1, \cdots, V_r$ are finite dimensional $\bf{k}$-vector spaces, we define an isomorphism
\begin{equation}\label{E:mulfus}
\mu_{V_1,\cdots, V_r}: \Det(V_1) \otimes \cdots \otimes \Det(V_r) \to \Det(V_1 \oplus \cdots \oplus V_r).
\end{equation}

\subsection{The isomorphism between the determinant lines of a $\mathbb{Z}_2$-graded complex and its cohomology }\label{SS:isomo}
For $k=0,1$, fix a direct sum decomposition
\begin{equation}\label{E:decompbha}
C^{\bar{k}}=B^{\bar{k}} \oplus H^{\bar{k}} \oplus A^{\bar{k}},
\end{equation}
such that $B^{\bar{k}}\oplus H^{\bar{k}}=(\Ker d_{\bar{k}}) \cap C^{\bar{k}}$ and $B^{\bar{k}}= d_{\overline{k+1}}\big(C^{\overline{k+1}}\big)=d_{\overline{k+1}}\big(A^{\overline{k+1}}\big)$. Then $H^{\bar{k}}$ is naturally isomorphic to the cohomology $H^{\bar{k}}(d_{\bar k})$ and $d_{\bar{k}}$ defines an isomorphism $d_{\bar{k}}: A^{\bar{k}} \to B^{\overline{k+1}}$.

Fix $c_{\bar{k}} \in \Det(C^{\bar{k}})$ and $a_{\bar{k}} \in \Det(A^{\bar{k}})$. Let $d_{\bar{k}}(a_{\bar{k}}) \in \Det(B^{\overline{k+1}})$ denote the image of $a_{\bar{k}}$ under the map $\Det(A^{\bar{k}}) \to \Det(B^{\overline{k+1}})$ induced by the isomorphism $d_{\bar{k}}:A^{\bar{k}} \to B^{\overline{k+1}}$. Then there is a unique element $h_{\bar{k}} \in \Det(H^{\bar{k}})$ such that
\begin{equation}\label{E:cbahce}
c_{\bar{k}}= \mu_{B^{\bar{k}},H^{\bar{k}},A^{\bar{k}}}\big(d_{\overline{k+1}}(a_{\overline{k+1}})\otimes h_{\bar{k}} \otimes a_{\bar{k}}\big),
\end{equation}
where $\mu_{B^{\bar{k}},H^{\bar{k}},A^{\bar{k}}}$ is the fusion isomorphism, cf. \eqref{E:mulfus}, see also \cite[Subsection 2.3]{BK3}.

Define the canonical isomorphism
\begin{equation}\label{E:isomorphism}
\phi_{C^\bullet}=\phi_{(C^{\bullet},d)} \,:\,\Det(C^{\bullet} ) \longrightarrow \Det(H^{\bullet}(d)),
\end{equation}
by the formula
\begin{equation}\label{E:phic}
\phi_{C^\bullet}: c_{\bar{0}} \otimes c_{\bar{1}}^{-1} \mapsto (-1)^{\mathcal{N}(C^{\bullet})} h_{\bar{0}} \otimes h_{\bar{1}}^{-1}, 
\end{equation}
where 
\begin{equation}\label{E:ncbullet}
\mathcal{N}(C^\bullet) \, := \, \frac{1}{2} \sum_{k=0,1} \dim A^{\bar k} \cdot \big(\dim A^{\bar k}+(-1)^{k+1}\big).
\end{equation}

\subsection{The fusion isomorphism for $\mathbb{Z}_2$-graded complexes}
Let $C^\bullet=C^{\bar 0} \oplus C^{\bar 1}$ and $\widetilde{C}^\bullet=\widetilde{C}^{\bar 0} \oplus \widetilde{C}^{\bar 1}$ be finite dimensional $\mathbb{Z}_2$-graded $\bf{k}$-vector spaces. The fusion isomorphism 
\[
\mu_{C^\bullet ,\widetilde{C}^{\bullet}} : \Det(C^{\bullet}) \otimes \Det(\widetilde{C}^{\bullet}) \to \Det(C^{\bullet} \oplus \widetilde{C}^{\bullet}),
\]
is defined by the formula
\begin{equation}\label{E:grafus}
\mu_{C^{\bullet},\widetilde{C}^{\bullet}} := (-1)^{\mathcal{M}(C^{\bullet},\widetilde{C}^{\bullet})}  \mu_{C^{\bar 0},\widetilde{C}^{\bar 0}} \otimes \mu_{C^{\bar 1},\widetilde{C}^{\bar 1}}^{-1},
\end{equation}
where 
\begin{equation}\label{E:mvw123}
\mathcal{M}(C^{ \bullet},\widetilde{C}^{ \bullet}) := \dim C^{\bar 1} \cdot \dim \widetilde{C}^{\bar 0}.
\end{equation}

The following lemma is a $\mathbb{Z}_2$-graded analogue of \cite[Lemma 2.7]{BK3} and \cite[Lemma 2.4]{FT2}. The proof is a slight modification of the proof of \cite[Lemma 2.7]{BK3}.

\begin{lemma}\label{L:commu}
Let $(\,C^\bullet, d\,)$ and $(\,\widetilde{C}^\bullet, \widetilde{d}\,)$ be $\mathbb{Z}_2$-graded complexes with finite dimensional $\bf{k}$-vector spaces. Further, assume that the Euler characteristics $\chi(C^\bullet)=\chi(\widetilde{C}^\bullet)=0$. Then the following diagram commutes:
\begin{equation}\label{E:comm}
\begin{CD}
\Det(C^\bullet) \otimes \Det(\widetilde{C}^\bullet)     & @>\phi_{C^\bullet}\otimes \phi_{\widetilde{C}^\bullet}>> & \Det\big(H^\bullet(d)\big) \otimes \Det\big(\, H^\bullet(\widetilde{d})\,\big)     \\
  @V \mu_{C^\bullet,\widetilde{C}^\bullet} VV &        & @VV  \mu_{H^\bullet(d),H^\bullet(\widetilde{d})} V   \\
  \Det(C^\bullet \oplus \widetilde{C}^\bullet)     & @> \phi_{C^\bullet \oplus \widetilde{C}^\bullet} >> & \Det\big(\, H^\bullet(d \oplus \widetilde{d})\, \big) \cong \Det\big( \, H^\bullet(d) \oplus H^\bullet(\widetilde{d})\, \big)
\end{CD}
\end{equation}
\end{lemma}
\begin{proof}
Proceed similar procedures as the proof of Lemma 2.7 of \cite{BK3}, cf. \cite[P. 152-153]{BK3}, we conclude that to prove the commutativity of the diagram \eqref{E:comm} it remains to show that, mod $2$,
\begin{equation}\label{E:eukccnmm}
\mathcal{N}(C^\bullet \oplus \widetilde{C}^\bullet)+\mathcal{N}(C^\bullet)+\mathcal{N}( \widetilde{C}^\bullet)+\mathcal{M}(C^\bullet, \widetilde{C}^\bullet) +\mathcal{M}(H^\bullet, \widetilde{H}^\bullet) 
\end{equation}
\[
\equiv \sum_{k=0,1}\big( \dim A^{\bar k} \cdot \dim \widetilde{A}^{\overline{k+1}}+ \dim H^{\bar k} \cdot \dim \widetilde{A}^{\overline{k+1}}+\dim A^{\bar k} \cdot \dim \widetilde{H}^{\bar k}\big).
\]
Using the identity
\begin{equation}\label{E:usid1}
\frac{(x+y)(x+y+(-1)^j)}{2}-\frac{x(x+(-1)^j)}{2}-\frac{y(y+(-1)^j)}{2}=xy,
\end{equation}
where $x,y \in \mathbb{C}, j \in \mathbb{Z}_{\ge 0}$, we have
\begin{equation}\label{E:nnn}
\mathcal{N}(C^\bullet \oplus \widetilde{C}^\bullet) - \mathcal{N}(C^\bullet)- \mathcal{N}( \widetilde{C}^\bullet) = \sum_{k=0,1} \dim A^{\bar k} \cdot \dim \widetilde{A}^{\bar k}.
\end{equation}
By \eqref{E:decompbha} and the equalities $\dim A^{\overline{k+1}}= \dim B^{\bar k}$, $\dim \widetilde{A}^{\overline{k+1}}= \dim \widetilde{B}^{\bar k}$, we have
\begin{equation}\label{E:cbha1}
\dim C^{\bar k} = \dim A^{\bar k} + \dim A^{\overline{k+1}} + \dim H^{\bar k}, \quad \dim \widetilde{C}^{\bar k} = \dim \widetilde{A}^{\bar k} + \dim \widetilde{A}^{\overline{k+1}} + \dim \widetilde{H}^{\bar k}.
\end{equation}
By \eqref{E:mvw123}, \eqref{E:cbha1} and a straightforward computation, we obtain, modulo $2$,
\begin{equation}\label{E:lonequ}
\begin{array}{l}
\mathcal{M}(C^\bullet, \widetilde{C}^\bullet) +\mathcal{M}(H^\bullet, \widetilde{H}^\bullet) + \sum_{k=0,1}\big( \dim A^{\bar k} \cdot \dim \widetilde{A}^{\overline{k+1}}+ \dim H^{\bar k} \cdot \dim \widetilde{A}^{\overline{k+1}}+\dim A^{\bar k} \cdot \dim \widetilde{H}^{\bar k}\big)\\

\\
= \sum_{k=0,1}\dim A^{\bar k} \cdot \dim \widetilde{A}^{\bar k} + \dim A^{\bar 1}\cdot (\dim \widetilde{H}^{\bar 0}+ \dim \widetilde{H}^{\bar 1}) + (\dim H^{\bar 0}+ \dim H^{\bar 1})\cdot \dim \widetilde{A}^{\bar 1}.
 \end{array}
\end{equation}
By \eqref{E:nnn}, \eqref{E:lonequ} and the assumption that the Euler characteristic of the complex $(C^\bullet,d)$ (resp. $(\widetilde{C}^\bullet,\widetilde{d})$) is zero, i.e. $\sum_{k=0,1} \dim H^{\bar k} \equiv 0 \, (\operatorname{mod}\, 2)$ (resp. $\sum_{k=0,1}\dim \widetilde{H}^{\bar k} \equiv 0 (\operatorname{mod}\, 2) \,$), we obtain the equality \eqref{E:eukccnmm}.
\end{proof}

\subsection{The refined torsion of a finite dimensional $\mathbb{Z}_2$-graded complex with a chirality operator}
 

Let $(C^\bullet,d)$ be a $\mathbb{Z}_2$-graded complex defined as \eqref{E:detline}. A {\em chirality operator} is an involution $\Gamma : C^\bullet \to C^\bullet$ such that $\Gamma(C^{\bar{k}})=C^{\overline{k+1}}, k=0,1$. For $c_{\bar{k}} \in \Det(C^{\bar{k}})$, we denote by $\Gamma c_{\bar{k}} \in \Det(C^{\overline{k+1}})$ the image of $c_{\bar{k}}$ under the isomorphism $\Det(C^{\bar{k}}) \to \Det(C^{\overline{k+1}})$ induced by $\Gamma$.

Fix a nonzero element $c_{\bar{0}} \in \Det(C^{\bar{0}})$ and consider the element 
\begin{equation}\label{E:elec}
c_{\Gamma}:=(-1)^{\mathcal{R}(C^\bullet)} \cdot c_{\bar{0}} \otimes (\Gamma c_{\bar{0}})^{-1} \in \Det(C^\bullet),
\end{equation} 
where
\begin{equation}\label{E:rcbullet}
\mathcal{R}(C^\bullet ) \, := \, \frac{1}{2} \dim C^{\bar 0} \cdot (\dim C^{\bar 0}+1).
\end{equation}

The element defined in \eqref{E:elec} is a $\mathbb{Z}_2$-graded analogue of the $\mathbb{Z}$-graded one as defined in \cite[(4-1)]{BK3}, by Braverman-Kappeler, and is chosen to fit the $\mathbb{Z}_2$-graded setting.

\begin{definition}\label{D:elerho1}
The {\em refined torsion} of the pair $(C^\bullet,\Gamma)$ is the element
\begin{equation}\label{E:elerho}
\rho_{\Gamma}=\rho_{C^{\bullet},\Gamma}:= \Phi_{C^{\bullet}}(c_{\Gamma}),
\end{equation}
where $\Phi_{C^{\bullet}}$ is the canonical map defined by \eqref{E:isomorphism}.
\end{definition}


The following is the $\mathbb{Z}_2$-graded analogue of Lemma 4.7 of \cite{BK3}.
\begin{lemma}
Let $(C^\bullet,d)$ and $(\widetilde{C}^\bullet,\widetilde{d})$ be $\mathbb{Z}_2$-graded complexes defined as \eqref{E:detline} and let $\Gamma : C^\bullet \to C^\bullet$, \, $\widetilde{\Gamma} : \widetilde{C}^\bullet \to \widetilde{C}^\bullet$ be chirality operators. Then $\widehat{\Gamma}:=\Gamma \oplus \widetilde{\Gamma} : C^\bullet \oplus \widetilde{C}^\bullet \to C^\bullet \oplus \widetilde{C}^\bullet$ is a chirality operator on the direct sum complex $(C^\bullet \oplus \widetilde{C}^\bullet ,d \oplus \widetilde{d})$ and
\begin{equation}\label{E:disuto}
\rho_{\widehat{\Gamma}}=\mu_{H^\bullet(d),H^\bullet(\widetilde{d})}\big( \rho_\Gamma \otimes \rho_{\widetilde{\Gamma}} \big).
\end{equation}
\end{lemma}
\begin{proof}
Clearly, $\widehat{\Gamma}^2=1$ and $\widehat{\Gamma}(C^{\bar k} \oplus \widetilde{C}^{\bar k})=C^{\overline{k+1}}\oplus \widetilde{C}^{\overline{k+1}}$. Hence, $\widehat{\Gamma}$ is a chirality operator. By Lemma \ref{L:commu}, to prove \eqref{E:disuto} it is enough to show that
\begin{equation}\label{E:maidp}
c_{\widehat{\Gamma}}=\mu_{C^\bullet,\widehat{C}^\bullet}\big( c_\Gamma \otimes c_{\widetilde{\Gamma}} \big).
\end{equation}
Fix nonzero elements $c_{\bar 0} \in \Det(C^{\bar 0}), \widetilde{c}_{\bar 0} \in \Det(\widetilde{C}^{\bar 0})$ and set $\widehat{c}_{\bar 0} = \mu_{C^{\bar 0},\widetilde{C}_{\bar 0}}(c_{\bar 0} \otimes \widetilde{c}_{\bar 0})$. We denote the operators induced by $\Gamma$ and $\widetilde{\Gamma}$ on $\Det(C^\bullet)$ and $\Det(\widetilde{C}^\bullet)$ by the same letters. Thus,
\[
\widehat{\Gamma}\, \widehat{c}_{\bar 0} = (\Gamma \oplus \widetilde{\Gamma}) \circ \mu_{C^{\bar 0}, \widetilde{C}_{\bar 0}}(c_{\bar 0} \otimes \widetilde{c}_{\bar 0}) = \mu_{C^{\bar 1},\widetilde{C}^{\bar 1}}(\Gamma c_{\bar 0} \otimes \widetilde{\Gamma} \widetilde{c}_{\bar 0}).
\]
Hence, it follows from \eqref{E:grafus} and \eqref{E:elec} that
\begin{align}\label{E:watd}
\mu_{C^\bullet,\widetilde{C}^\bullet}(c_\Gamma \otimes c_{\widetilde{\Gamma}}) & = (-1)^{\mathcal{M}(C^\bullet, \widetilde{C}^\bullet)+\mathcal{R}(C^\bullet)+\mathcal{R}(\widetilde{C}^\bullet)} \cdot \widehat{c}_{\bar 0} \otimes (\widehat{\Gamma} \, \widehat{c}_{\bar 0})^{-1} \nonumber \\
& = (-1)^{\mathcal{M}(C^\bullet, \widetilde{C}^\bullet)+\mathcal{R}(C^\bullet)+\mathcal{R}(\widetilde{C}^\bullet)-\mathcal{R}(C^\bullet \oplus \widetilde{C}^\bullet)} \cdot c_{\widehat{\Gamma}}.
\end{align}
Using the identity \eqref{E:usid1}, we obtain from \eqref{E:rcbullet}
\begin{equation}\label{E:rrmr}
\mathcal{R}(C^\bullet)+\mathcal{R}(\widetilde{C}^\bullet)-\mathcal{R}(C^\bullet \oplus \widetilde{C}^\bullet) = \dim C^{\bar 0} \cdot \dim \widetilde{C}^{\bar 0}.
\end{equation}
Using the isomorphism $\Gamma : C^{\bar 0} \to C^{\bar 1}$, one sees that $\dim C^{\bar 0} = \dim C^{\bar 1}$. Combining this fact with \eqref{E:mvw123} and \eqref{E:rrmr}, we conclude that 
\begin{equation}\label{E:id13}
\mathcal{M}(C^\bullet, \widetilde{C}^\bullet)+\mathcal{R}(C^\bullet)+\mathcal{R}(\widetilde{C}^\bullet)-\mathcal{R}(C^\bullet \oplus \widetilde{C}^\bullet) \equiv 0 \quad \text{mod} \ 2.
\end{equation}
The identity \eqref{E:maidp} follows from \eqref{E:watd} and \eqref{E:id13}.
\end{proof}

\subsection{Dependence of the $\mathbb{Z}_2$-graded refined torsion on the chirality operator}

Suppose that $\Gamma_t,t \in \mathbb{R}$, is a smooth family of chirality operators on the $\mathbb{Z}_2$-graded complex $(C^{  \bullet},d)$. Let $\dot{\Gamma}_t:C^{\bar k} \to C^{\overline{k+1}},k=0,1$, denote the derivative of $\Gamma_t$ with respect to $t$. Then, for $k=0,1$, the composition $\dot{\Gamma_t}\circ \Gamma_t$ maps $C^{\bar k}$ into itself. Define the {\em supertrace} $\operatorname{Tr}_s(\dot{\Gamma_t}\circ \Gamma_t)$ of $\dot{\Gamma_t}\circ \Gamma_t$ by the formula
\begin{equation}\label{E:supertrace}
\operatorname{Tr}_s(\dot{\Gamma_t}\circ \Gamma_t):= \operatorname{Tr}(\dot{\Gamma_t}\circ \Gamma_t|_{C^{\bar 0}})- \operatorname{Tr}(\dot{\Gamma_t}\circ \Gamma_t|_{C^{\bar 1}}).
\end{equation}

The following proposition is the $\mathbb{Z}_2$-graded analogue of Proposition 4.9 of \cite{BK3}. We modify the proof of Proposition 4.9 of \cite{BK3} slightly to fit our setting.
\begin{proposition}
Let $(C^{  \bullet},d)$ be a $\mathbb{Z}_2$-graded complex of finite dimensional $\bf{k}$-vector spaces and let $\Gamma_t : C^{\bar{k}} \to C^{\overline{k+1}}, t \in \mathbb{R}$, be a smooth family of chirality operators on $C^{  \bullet}$. Then the following equality holds
\begin{equation}\label{E:varfor}
\frac{d}{dt}\rho_{\Gamma_t}= \frac{1}{2} \operatorname{Tr}_s(\dot{\Gamma_t}\circ \Gamma_t) \cdot \rho_{\Gamma_t}.
\end{equation}
\end{proposition}
\begin{proof}
Let $\Gamma_{t,\bar{0}}$ denote the restriction of $\Gamma_t$ to $C^{\bar 0}$. We denoted the map $\Det(C^{\bar 0}) \to \Det(C^{\bar{1}})$ induced by $\Gamma_t$ by the same symbol $\Gamma_t$ above. To avoid confusion we denote this map by $\Gamma^{\Det}_{t,\bar{0}}$ in the proof.

For $ t_0 \in \mathbb{R}$, we have $\Gamma_{t,\bar{0}}=\Gamma_{t,\bar{0}} \circ \Gamma_{t_0,\bar{1}}\Gamma_{t_0,\bar{0}}$. 
Hence, 
\[
\frac{d}{dt}\Big|_{t=t_0} \Gamma_{t,\bar{0}}^{\Det} = \frac{d}{dt} \Big|_{t=t_0} \Big[ \Det(\Gamma_{t,\bar{0}} \circ \Gamma_{t_0,\bar{1}}) \Gamma_{t_0,\bar{0}}^{\Det} \Big] = \operatorname{Tr}( \dot{\Gamma}_{t_0, \bar{0} }  \circ \Gamma_{t_0,\bar{1}})\Gamma^{\Det}_{t_0,\bar{0}},
\]
where for the latter equality we used the fact that for any smooth family of operators $A_t:C^{\bar{1}} \to C^{\bar{1}}$, one has $\frac{d}{dt} \Det(A_t)=\operatorname{Tr}(\dot{A_t}A_t^{-1})\cdot \Det(A_t)$ and that $\Gamma_{t_0,\bar{0}}^{-1}=\Gamma_{t_0,\bar{1}}$. Hence, for any nonzero element $c_{\bar 0} \in \Det(C^{\bar 0})$, we have
\begin{equation}\label{E:varif}
\frac{d}{dt}(\Gamma_{t,\bar{0}}^{\Det}(c_{\bar 0}))^{\pm}= \pm \operatorname{Tr}(\dot{\Gamma}_{t,\bar{0}} \circ \Gamma_{t,\bar{1}})\cdot (\Gamma^{\Det}_{t,\bar{0}})^{\pm}.
\end{equation}
By \eqref{E:varif} and the definition \eqref{E:elec} of $c_\Gamma$, we obtain
\begin{equation}\label{E:varform}
\frac{d}{dt}c_{\Gamma_t}=- \operatorname{Tr}(\dot{\Gamma}_{t,\bar{0}}\circ \Gamma_{t,\bar{1}}) \cdot c_{\Gamma_t}.
\end{equation} 
Since $\Gamma_{t,\bar{0}} \circ \Gamma_{t,\bar{1}}=1$, we have
\[
0=\frac{d}{dt} \operatorname{Tr}(\Gamma_{t,\bar{0}} \circ \Gamma_{t,\bar{1}})= \operatorname{Tr}(\dot{\Gamma}_{t,\bar{0}} \circ \Gamma_{t,\bar{1}})+\operatorname{Tr}(\Gamma_{t,\bar{0}} \circ \dot{\Gamma}_{t,\bar{1}}).
\]
Hence,
\begin{equation}\label{E:trpm}
\operatorname{Tr}(\dot{\Gamma}_{t,\bar{0}} \circ \Gamma_{t,\bar{1}}) = - \operatorname{Tr}(\dot{\Gamma}_{t,\bar{1}}  \circ \Gamma_{t,\bar{0}}).
\end{equation}
Combining \eqref{E:varform} with \eqref{E:trpm}, we obtain \eqref{E:varfor}.
\end{proof}


\subsection{The signature operator}
We now introduce the $\mathbb{Z}_2$-graded analogue of the $\mathbb{Z}$-graded finite dimensional odd signature operator of \cite[Section 5]{BK3}.
The {\em signature operators} $\B_{\overline{k}},k=0,1$ are defined by the formula
\begin{equation}\label{E:sigop}
\B_{\bar{k}} := \Gamma d_{\bar{k}} + d_{\overline{k+1}} \Gamma. 
\end{equation}
Define
\begin{equation}\label{E:cjpm}
C^{\bar{k}}_+ := \Ker(d_{\overline{k+1}} \circ \Gamma) \cap C^{\bar{k}} = \Gamma(\Ker d_{\overline{k+1}} \cap C^{\overline{k+1}}), \qquad
C^{\bar{k}}_- := \Ker d_{\bar{k}} \cap C^{\bar{k}}. 
\end{equation}
Let $\B^{\pm}_{\bar{k}}$ denote the restriction of $\B_{\bar{k}}$ to $C^{\bar{k}}_{\pm}$. Then one has 
\begin{equation}\label{E:imbp}
\im \B^+_{\bar{k}} \subseteq \im (\Gamma \circ d_{\bar{k}}|_{C^{\bar{k}}}) \subseteq \Gamma(\Ker d_{\overline{k+1}}|_{C^{\overline{k+1}}}) \subseteq C^{\bar{k}}_+ ;
\end{equation}
\begin{equation}\label{E:imbm}
\im \B^-_{\bar{k}} \subseteq \im ( d_{\overline{k+1}} \circ\Gamma|_{C^{\bar{k}}}) \subseteq \im (d_{\overline{k+1}}|_{C^{\overline{k+1}}}) \subseteq C^{\bar{k}}_-.
\end{equation}
Hence,
$$
\B^+_{\bar{k}}=\Gamma \circ d_{\bar{k}} : C^{\bar{k}}_+ \to C^{\bar{k}}_+, \quad \B^-_{\bar{k}}=d_{\overline{k+1}} \circ \Gamma : C^{\bar{k}}_- \to C^{\bar{k}}_-.
$$
Note thar $\B_{\bar{k}}=\Gamma  \circ \B_{\overline{k+1}} \circ \Gamma$.

The following lemma is the $\mathbb{Z}_2$-graded analogue of \cite[Lemma 5.2]{BK3}. The proof is a verbatim repetition of the proof of \cite[Lemma 5.2]{BK3}, we skip the proof.
\begin{lemma}\label{L:ac}
Suppose that the signature operators $\B_{\bar{k}},k=0,1$ are bijective. Then the complex $(C^{\bullet},d)$ is acyclic and
\begin{equation}\label{E:cbij}
C^{\bar{k}}=C^{\bar{k}}_+ \oplus C^{\bar{k}}_-.
\end{equation}
\end{lemma}

\subsection{Calculation of the refined torsion in case $\B$ is bijective}
In this subsection we compute the $\mathbb{Z}_2$-graded refined torsion in the case that $\B_{\bar{k}},k=0,1$ are bijective. Assume that the signature operators $\B_{\bar{k}},k=0,1$ are bijective. Then, by Lemma~\ref{L:ac}, the complex $(C^{\bullet},d)$ is acyclic. Note that $\Gamma \B^-_{\bar 0}\Gamma=\B^+_{\bar 1}$. Hence $\Det(\B^-_{\bar 0})=\Det(\B^+_{\bar 1})$. Then we have the following definition.

\begin{definition}
The {\em graded determinant} of the signature operator $\B_{\bar{0}}$ is defined by the formula
\begin{equation}\label{E:figraddet}
\Det_{\operatorname{gr}}(\B_{\bar{0}}):=\Det(\B^+_{\bar{0}})/\Det(-\B^-_{\bar{0}})=\Det(\B^+_{\bar{0}})/\Det(-\B^+_{\bar{1}}).
\end{equation}
\end{definition}

The following proposition is the $\mathbb{Z}_2$-graded analogue of \cite[Proposition 5.6]{BK3}. We modify the proof of \cite[Proposition 5.6]{BK3} slightly to fit our setting.
\begin{proposition}\label{P:rhogddet1}
Suppose that the signature operators $\B_{\bar k},k=0,1$ are invertible and, hence, the complex $(C^{\bullet},d)$ is acyclic. Then 
\begin{equation}\label{E:rhogddet}
\rho_{\Gamma}=\Det_{\operatorname{gr}}(\B_{\bar{0}}).
\end{equation}
\end{proposition}
\begin{proof}
We choose the decomposition \eqref{E:decompbha} to be $C^{\bar{k}}=C^{\bar{k}}_- \oplus C^{\bar{k}}_+$ and define elements $c_{\bar{k}}$ as follows. Fix a nonzero element $a_{\bar{k}} \in \Det(C^{\bar{k}}_+)$ and set
\[
c_{\bar{k}}= \mu_{C^{\bar{k}}_-,C^{\bar{k}}_+}(\Gamma a_{\overline{k+1}} \otimes a_{\bar{k}}),
\]
where $\mu_{C^{\bar{k}}_-,C^{\bar{k}}_+}$ is the fusion isomorphism, cf. \eqref{E:fusio}, see also \cite[(2-5)]{BK3}.
Note that, by \eqref{E:fusio1},
\begin{align}
\Gamma c_{\bar{k}} & =  \mu_{C^{\overline{k+1}}_+,C^{\overline{k+1}}_-}(a_{\overline{k+1}} \otimes \Gamma a_{\bar{k}})\nonumber \\
&= (-1)^{\dim C^{\overline{k+1}}_+ \cdot \dim C^{\overline{k+1}}_-} \cdot \mu_{C^{\overline{k+1}}_-,C^{\overline{k+1}}_+}(\Gamma a_{\bar{k}} \otimes a_{\overline{k+1}}) \nonumber \\
&= (-1)^{\dim C^{\overline{k+1}}_+ \cdot \dim C^{\overline{k+1}}_-} \cdot c_{\overline{k+1}}.
\end{align}
Thus, from \eqref{E:elec}, we obtain
\begin{align}\label{E:cgs}
c_{\Gamma} & =  (-1)^{\mathcal{R}(C^{\bullet})} \cdot c_{\bar{0}} \otimes (\Gamma c_{\bar{0}})^{-1} \nonumber \\
 & =  (-1)^{\mathcal{R}(C^{\bullet})+ \dim C^{\bar{1}}_+ \cdot \dim C^{\bar{1}}_-} \cdot c_{\bar{0}} \otimes c_{\bar{1}}^{-1}.
\end{align}
Hence, by \eqref{E:elerho} and \eqref{E:phic}, to compute $\rho_{\Gamma}$ we need to compute the elements $h_{\bar{k}} \in \Det(H^{\bar{k}}) \cong \bf{k}$.

If $L$ is a complex line and $x,y \in L$ with $y \not=0$, we denote by $[x:y] \in \bf{k}$ the unique number such that $x=[x:y]y$. Then
\begin{align}\label{E:hrel}
h_{\bar{k}} & =  [c_{\bar{k}}: \mu_{C^{\bar{k}}_-,C^{\bar{k}}_+}(d_{\overline{k+1}}a_{\overline{k+1}} \otimes a_{\bar{k}})] \nonumber \\
&= [\mu_{C^{\bar{k}}_-,C^{\bar{k}}_+}(\Gamma a_{\overline{k+1}} \otimes a_{\bar{k}}):\mu_{C^{\bar{k}}_-,C^{\bar{k}}_+}(d_{\overline{k+1}}a_{\overline{k+1}} \otimes a_{\bar{k}})] \nonumber \\
& =  [\Gamma a_{\overline{k+1}} : d_{\overline{k+1}}a_{\overline{k+1}}] \nonumber \\
&= [a_{\overline{k+1}}: \Gamma d_{\overline{k+1}} a_{\overline{k+1}}]\nonumber \\
& = \Det (\Gamma d_{\overline{k+1}})^{-1}.
\end{align}
Combining \eqref{E:phic} and \eqref{E:hrel}, we obtain
\begin{align}\label{E:phccc}
\Phi_{C^{\bullet}} ( c_{\bar{0}} \otimes c_{\bar{1}}^{-1} ) & = (-1)^{\mathcal{N}(C^{\bullet})} \cdot \Det(\Gamma d_{\overline{0}})\cdot \Det (\Gamma d_{\overline{1}})^{-1} \nonumber \\
& = (-1)^{\mathcal{N}(C^{\bullet}) + \dim C^{\bar{1}}_+} \cdot \Det (\Gamma d_{\overline{0}})\cdot \Det (- \Gamma d_{\overline{1}})^{-1}. 
\end{align}
Combining \eqref{E:elerho}, \eqref{E:figraddet}, \eqref{E:cgs} and \eqref{E:phccc}, we have
\[
\rho_{\Gamma}= \Phi_{C^{\bullet}}(c_{\Gamma}) = (-1)^{\mathcal{R}(C^{\bullet})+ \dim C^{\bar{1}}_+ \cdot \dim C^{\bar{1}}_- + \mathcal{N}(C^{\bullet}) + \dim C^{\bar{1}}_+} \cdot \Det_{\operatorname{gr}}(\B_{\bar{0}}).
\]
Hence, we are remaining to show that
\begin{equation}\label{E:fcbullet}
\mathcal{F}(C^{  \bullet}):=\mathcal{R}(C^{\bullet})+ \dim C^{\bar{1}}_+ \cdot \dim C^{\bar{0}}_+ + \mathcal{N}(C^{ {\bullet}}) + \dim C^{\bar{1}}_+ \equiv 0 \quad \text{mod} \ 2,
\end{equation}
here we use the fact that $\Gamma \B^+_{\bar{0}} \Gamma = \B^-_{\bar{1}}$.
By the fact that $\Gamma \B^+_{\bar{1}} \Gamma = \B^-_{\bar{0}}$ and \eqref{E:rcbullet}, we have 
\begin{align}\label{E:rcbullet1}
\mathcal{R}(C^{\bullet}) &=  \frac{1}{2}(\dim C^{\bar 0}_+ + \dim C^{\bar 0}_-)\cdot(\dim C^{\bar 0}_+ + \dim C^{\bar 0}_- +1)\nonumber \\
 & =  \frac{1}{2}(\dim C^{\bar 0}_+ + \dim C^{\bar 1}_+)\cdot (\dim C^{\bar 0}_+ + \dim C^{\bar 1}_+ +1) \nonumber \\
 & =  \frac{1}{2}(\dim C^{\bar 0}_+)^2 + \frac{1}{2}(\dim C^{\bar 1}_+)^2 + \dim C^{\bar 0}_+ \cdot \dim C^{\bar 1}_+ + \frac{1}{2} \dim C^{\bar 0}_+ + \frac{1}{2} \dim C^{\bar 1}_+.
\end{align}
Recall that, \eqref{E:ncbullet},
\begin{align}\label{E:ncbullet1}
\mathcal{N}(C^{\bullet}) & =  \frac{1}{2}\dim C^{\bar 0}_+(C^{\bar 0}_+ +1) + \frac{1}{2}\dim C^{\bar 1}_+(C^{\bar 1}_+ -1) \nonumber \\ & =  \frac{1}{2}(\dim C^{\bar 0}_+)^2 + \frac{1}{2} \dim C^{\bar 0}_+ +  \frac{1}{2}(\dim C^{\bar 1}_+)^2 - \frac{1}{2} \dim C^{\bar 1}_+.
\end{align}
Combining \eqref{E:fcbullet}, \eqref{E:rcbullet1} and \eqref{E:ncbullet1}, we have
\begin{equation}\label{E:fcbullet1}
\mathcal{F}(C^{  \bullet})=(\dim C^{\bar 0}_+)^2 + (\dim C^{\bar 1}_+)^2 + 2\dim C^{\bar 0}_+ \cdot \dim C^{\bar 1}_+ +  \dim C^{\bar 0}_+ + \dim C^{\bar 1}_+. 
\end{equation}
By \eqref{E:fcbullet1} and the fact that for any $x \in \mathbb{Z}, x(x+1) \equiv 0\, (\text{mod}\, 2)$, we obtain
\[
\mathcal{F}(C^{  \bullet})=0.
\]
\end{proof}

\subsection{Calculation of the refined torsion in case $\B$ is not bijective}
In this subsection we compute the $\mathbb{Z}_2$-graded refined torsion in the case that $\B_{\bar{k}},k=0,1$ are not bijective. Note that the operator $\B^2_{\bar{k}}$ maps $C^{\bar{k}}$ into itself. For an arbitrary interval $\mathcal{I}$, denote by $C^{\bar{k}}_{\mathcal{I}} \subset C^{\bar{k}}$ the linear span of the generalized eigenvectors of the restriction of $\B^2_{\bar{k}}$ to $C^{\bar{k}}$, corresponding to eigenvalue $\lambda$ with $\lambda \in \mathcal{I}$. Since both $\Gamma$ and $d_{\bar{k}}$ commute with $\B_{\bar{k}}$ (and, hence, with $\B^2_{\bar{k}}$), $\Gamma (C^{\bar{k}}_{\mathcal{I}}) \subset C^{\overline{k+1}}_{\mathcal{I}}$ and $d_{\bar{k}}(C^{\bar{k}}_{\mathcal{I}}) \subset C^{\overline{k+1}}_{\mathcal{I}}$. Hence, we obtain a subcomplex $C^{\bullet}_{\mathcal{I}}$ of $C^{\bullet}$ and the restriction $\Gamma_{\mathcal{I}}$ of $\Gamma$ to $C^{\bullet}_{\mathcal{I}}$ is a chirality operator for $C^{\bullet}_{\mathcal{I}}$. We denote by $H^{\bullet}_{\mathcal{I}}(d)$ the cohomology of the complex $(C^\bullet_{\mathcal{I}},d_{\mathcal{I}})$. Denote by $d_{\bar{k},\mathcal{I}}$ and $\B_{\bar{k},\mathcal{I}}$ the restrictions of $d_{\bar{k}}$ and $\B_{\bar{k}}$ to $C^{\bar{k}}_{\mathcal{I}}$. Then $\B_{\bar{k},\mathcal{I}} = \Gamma_{\mathcal{I}} d_{\bar{k},\mathcal{I}}+d_{\overline{k+1},\mathcal{I}}\Gamma_{\mathcal{I}}$.

\begin{lemma}
If $0 \notin \mathcal{I}$, then the complex $(C^{\bullet}_{\mathcal{I}},d_{\mathcal{I}})$ is acyclic.
\end{lemma}
\begin{proof}
If, for $k=0,1$, $x \in \Ker d_{\bar{k},\mathcal{I}}$, then $\B^2_{\bar{k},\mathcal{I}}x = (d_{\overline{k+1}}\Gamma)^2x \in \im d_{\overline{k+1},\mathcal{I}} \subset \Ker d_{\overline{k},\mathcal{I}}$. Since the operators $\B^2_{\bar{k},\mathcal{I}}:C^{\bar{k}}_{\mathcal{I}} \to C^{\bar{k}}_{\mathcal{I}},k=0,1$ are invertible, we conclude that $\Ker d_{\overline{k},\mathcal{I}} = \im d_{\overline{k+1},\mathcal{I}}$.
\end{proof}

For each $\lambda \ge 0$, $C^{\bullet}=C^{\bullet}_{[0,\lambda]} \oplus C^{\bullet}_{(\lambda,\infty)}$ and $H^{\bullet}(d)=0$ whereas $H^{\bullet}_{[0,\lambda]}(d) \cong H^{\bullet}(d)$. Hence there are canonical isomorphisms
\[
\Phi_\lambda : \Det(H^{\bullet}_{(\lambda,\infty)}(d)) \to \mathbb{C}, \qquad \Psi_{\lambda}: \Det(H^{\bullet}_{[0,\lambda]}(d)) \to \Det(H^{\bullet}(d)).
\]
In the sequel, we will write $t$ for $\Phi_\lambda(t) \in \mathbb{C}$. 

The following proposition is $\mathbb{Z}_2$-graded analogue of \cite[Proposition 5.10]{BK3}.
\begin{proposition}\label{P:indedec}
Let $(C^{\bullet},d)$ be a $\mathbb{Z}_2$-graded complex of finite dimensional $\bf{k}$-vector spaces and let $\Gamma$ be a chirality operator on $C^{\bullet}$. Then, for each $\lambda \ge 0$,
\[
\rho_{\Gamma}= \Det_{\operatorname{gr}}(\B_{(\lambda, \infty)}^{\bar{0}}) \cdot \rho_{\Gamma_{[0,\lambda]}},
\]
where we view $\rho_{\Gamma_{[0,\lambda]}}$ as an element of $\Det(H^{\bullet}(d))$ via the canonical isomorphism $\Psi_{\lambda}: \Det(H^{\bullet}_{[0,\lambda]}(d)) \to \Det(H^{\bullet}(d))$.
\end{proposition}
\begin{proof}
Recall the natural isomorphism 
\begin{equation}\label{E:isofu}
\Det(H^{\bar{k}}_{[0,\lambda]}(d) \otimes H^{\bar{k}}_{(\lambda,\infty)}(d)) \cong \Det(H^{\bar{k}}_{[0,\lambda]}(d) \oplus H^{\bar{k}}_{(\lambda,\infty)}(d))=\Det(H^{\bar{k}}(d))
\end{equation}
From Definition~\ref{D:elerho1}, Proposition~\ref{P:rhogddet1} and \eqref{E:isofu}, we obtain the result.
\end{proof}

\section{Graded determinant of the twisted odd signature operator}\label{S:gdtoso}
In this section we define the graded determinant of the odd signature operator, cf. \cite{APS2,Gi}, twisted by a flux form $H$, of a flat vector bundle $E$ over a closed oriented odd dimensional manifold $M$. We use this graded determinant to define an element of the determinant line of the twisted de Rham cohomology of the vector bundle $E$. We also study the relationship between this graded determinant and the $\eta$-invariant of the twisted odd signature operator.
 
\subsection{The twisted odd signature operator}
Let $M$ be a closed oriented smooth manifold of odd dimension $m=2r-1$ and let $E$ be a complex vector bundle over $M$ endowed with a flat connection $\nabla$. We denote by $\Omega^p(M,E)$ the space of $p$-forms with values in the flat bundle $E$, i.e., $\Omega^p(M,E)=\Gamma(\wedge^p(T^*M)_{\mathbb{R}}\otimes E)$ and by
\[
\nabla : \Omega^\bullet(M,E) \to \Omega^{\bullet+1}(M,E)
\]
the covariant differential induced by the flat connection on $E$. Fix a Riemannian metric $g^M$ on $M$ and let $\star : \Omega^\bullet(M,E) \to \Omega^{m-\bullet}(M,E)$ denote the Hodge $\star$-operator. We choose a Hermitian metric $h^E$ so that together with the Riemannian metric $g^M$ we can define a scalar product $<\cdot,\cdot>_M$ on $\Omega^{ \bullet}(M,E)$. Define the chirality operator $\Gamma = \Gamma(g^M):\Omega^\bullet(M,E) \to \Omega^\bullet(M,E)$ by the formula, cf. \cite[(7-1)]{BK3},
\begin{equation}\label{E:chirality}
\Gamma \omega := i^r(-1)^{\frac{q(q+1)}{2}} \star \omega, \quad \omega \in \Omega^q(M,E), 
\end{equation}
where $r$ given as above by $r=\frac{m+1}{2}$. The numerical factor in \eqref{E:chirality} has been chosen so that $\Gamma^2=\operatorname{Id}$, cf. Proposition 3.58 of \cite{BGV}.

Assume that $H$ is an odd degree closed differential form on $M$. Let $\Omega^{\bar{0}}(M,E):=\Omega^{\even}(M,E)$, $\Omega^{\bar{1}}(M,E):=\Omega^{\odd}(M,E)$ and $\nabla^H:=\nabla+H \wedge \cdot$. We assume that $H$ does not contain a 1-form component, which can be absorbed in the flat connection $\nabla$. 

\begin{definition}
The twisted odd signature operator is the operator
\begin{equation}\label{E:toso}
\B^H=\B(\nabla^H,g^M) := \Gamma \nabla^H + \nabla^H \Gamma : \Omega^{  \bullet}(M,E) \to \Omega^{  \bullet}(M,E).
\end{equation}
We denote by $\B^H_{\bar k}$ the restriction of $\B^H$ to the space $\Omega^{\bar k}(M,E),k=0,1$.
\end{definition}

\subsection{$\zeta$-function and $\zeta$-regularized determinant}
In this subsection we briefly recall some definitions of $\zeta$-regularized determinants of non self-adjoint elliptic operators. See \cite[Section 6]{BK3} for more details. Let $D:C^\infty(M,E) \to C^\infty(M,E)$ be an elliptic differential operator of order $n \ge 1$. Assume that $\theta$ is an Agmon angle, cf. for example, Definition 6.3 of \cite{BK3}. Let $\Pi: L^2(M,E) \to L^2(M,E)$ denote the spectral projection of $D$ corresponding to all nonzero eigenvalues of $D$. The $\zeta$-function $\zeta_\theta(s,D)$ of $D$ is defined as follows
\begin{equation}\label{E:zetdef}
\zeta_\theta(s,D)= \operatorname{Tr} \Pi D_\theta^{-s}, \quad \operatorname{Re} s > \frac{\dim M}{n}.
\end{equation} It was shown by Seeley \cite{Se} (See also \cite{Sh}) that $\zeta_\theta(s,D)$ has a meromorphic extension to the whole complex plane and that $0$ is a regular value of $\zeta_\theta(s,D)$.

\begin{definition}
The $\zeta$-regularized determinant of $D$ is defined by the formula
\[
\Det'_\theta(D):= \exp \Big(\, -\frac{d}{ds}\Big|_{s=0}\zeta_\theta(s,D)\,\Big). 
\]
\end{definition}
We denote by 
\[
\LDet'_\theta(D)= - \frac{d}{ds}\Big|_{s=0}\zeta_\theta(s,D).
\]

Let $Q$ be a $0$-th order pseudo-differential projection, ie. a $0$-th order pseudo-differential operator satisfying $Q^2=Q$. We set
\begin{equation}\label{E:zetadefpro}
\zeta_\theta(s,Q,D)= \operatorname{Tr} Q \Pi D_\theta^{-s}, \quad \operatorname{Re} s > \frac{\dim M}{n}.
\end{equation}
The function $\zeta_\theta(s,Q,D)$ also has a meromorphic extension to the whole complex plane and, by Wodzicki, \cite[Section 7]{Wo1}, it is regular at $0$.

\begin{definition}
Suppose that $Q$ is a $0$-th order pseudo-differential projection commuting with $D$. Then $V:=\operatorname{Im} Q$ is $D$ invariant subspace of $C^\infty(M,E)$. The $\zeta$-regularized determinant of the restriction $D|_V$ of $D$ to $V$ is defined by the formula
\[
\Det'_\theta(D|_V):= e^{\LDet'_\theta(D|_V)},
\]
where 
\begin{equation}\label{E:ldetv}
\LDet'_\theta(D|_V)=-\frac{d}{ds}\Big|_{s=0} \zeta_\theta(s,Q,D).
\end{equation}
\end{definition}

\begin{remark}
The prime in $\Det'_\theta$ and $\LDet'_\theta$ indicates that we ignore the zero eigenvalues of the operator in the definition of the regularized determinant. If the operator is invertible we usually omit the prime and write $\Det_\theta$ and $\LDet_\theta$ instead.
\end{remark}

\subsection{The graded determinant of the twisted odd signature operator}\label{SS:gdto}
Note that for each $k=0,1$, the operator $(\B^H)^2$ maps $\Omega^{\bar k}(M,E)$ to itself. Suppose that $\mathcal{I}$ is an interval of the form $[0,\lambda],(\lambda, \mu],$ or $(\lambda,\infty)$ $(\mu > \lambda \ge 0)$. Denote by $\Pi_{(\B^H)^2,\mathcal{I}}$ the spectral projection of $(\B^H)^2$ corresponding to the set of eigenvalues, whose absolute values lie in $\mathcal{I}$. set
\[
 \Omega^{  \bullet}_{\mathcal{I}}(M,E) := \Pi_{(\B^H)^2,\mathcal{I}}(\Omega^{  \bullet}(M,E)) \subset \Omega^{  \bullet}(M,E).
\]
If the interval $\mathcal{I}$ is bounded, then, cf. Section 6.10 of \cite{BK3}, the space $ \Omega^{  \bullet}_{\mathcal{I}}(M,E)$ is finite dimensional.

For each $k=0,1$, set
\begin{equation}\label{E:split}
\begin{array}{l}
\Omega^{\bar k}_{+,\mathcal{I}}(M,E) := \Ker(\nabla^H \Gamma) \cap  \Omega^{\bar k}_{\mathcal{I}}(M,E) = \big(\Gamma(\Ker \nabla^H) \big) \cap \Omega^{\bar k}_{\mathcal{I}}(M,E);\\

\\
\Omega^{\bar k}_{-,\mathcal{I}}(M,E) := \Ker(\Gamma \nabla^H ) \cap  \Omega^{\bar k}_{\mathcal{I}}(M,E) = \Ker \nabla^H \cap \Omega^{\bar k}_{\mathcal{I}}(M,E).
\end{array}
\end{equation}
Then
\begin{equation}\label{E:splitt1}
\Omega^{\bar k}_{\mathcal{I}}(M,E) = \Omega^{\bar k}_{+,\mathcal{I}}(M,E) \oplus \Omega^{\bar k}_{-,\mathcal{I}}(M,E) \quad \text{if} \quad 0 \notin \mathcal{I}.
\end{equation}
We consider the decomposition \eqref{E:splitt1} as a grading of the space $\Omega^{\bar k}_{\mathcal{I}}(M,E)$, and refer to $\Omega^{\bar k}_{+,\mathcal{I}}(M,E)$ and $\Omega^{\bar k}_{-,\mathcal{I}}(M,E)$ as the positive and negative subspaces of $\Omega^{\bar k}_{\mathcal{I}}(M,E)$. Denote by $\B^{H}_\mathcal{I}$ and $\B^{H}_{\bar{k},\mathcal{I}}$ the restrictions of $\B^H$ to the subspaces $\Omega^{  \bullet}_{\mathcal{I}}(M,E)$ and $\Omega^{\bar k}_{\mathcal{I}}(M,E)$ respectively. Then $\B^{H}_{\bar{k},\mathcal{I}}$ maps $\Omega^{\bar k}_{\pm,\mathcal{I}}(M,E)$ to itself. Let $\B^{H,\pm}_{\bar{k},\mathcal{I}}$ denote the restriction of $\B^{H}_{\bar{k},\mathcal{I}}$ to the subspace $\Omega^{\bar k}_{\pm,\mathcal{I}}(M,E)$. Clearly, the operator $\B^{H,\pm}_{\bar{k},\mathcal{I}}$ are bijective whenever $0 \notin \mathcal{I}$. Note that $\Gamma \B^{H,-}_{\bar{0},\mathcal{I}}\Gamma=\B^{H,+}_{\bar{1},\mathcal{I}}$. Hence $\Det_\theta(\B^{H,-}_{\bar{0},\mathcal{I}})=\Det_\theta(\B^{H,+}_{\bar{1},\mathcal{I}})$. Then we have the following definition.

\begin{definition}\label{D:graddet01}
Suppose that $0 \notin \mathcal{I}$. The graded determinant of the operator $\B^{H,\mathcal{I}}_{\bar 0}$ is defined by 
\begin{equation}\label{E:graddet}
\Det_{\gr,\theta}(\B^{H}_{\bar{0},\mathcal{I}}) \, := \, \frac{\Det_\theta(\B^{H,+}_{\bar{0},\mathcal{I}})}{\Det_\theta(-\B^{H,-}_{\bar{0},\mathcal{I}})} \, = \, \frac{\Det_\theta(\B^{H,+}_{\bar{0},\mathcal{I}})}{\Det_\theta(-\B^{H,+}_{\bar{1},\mathcal{I}})} \in \mathbb{C} \backslash \{ 0 \},
\end{equation}
where $\Det_\theta$ denotes the $\zeta$-regularized determinant associated to the Agmon angle $\theta \in (-\pi,0)$, cf. for example, Section 6 of \cite{BK3}.
\end{definition}
We define, cf. \eqref{E:ldetv},
\begin{equation}\label{E:ldetgr}
\LDet_{\gr,\theta}(\B^{H}_{\bar{0},\mathcal{I}})=\LDet_\theta(\B^{H,+}_{\bar{0},\mathcal{I}})-\LDet_\theta(-\B^{H,-}_{\bar{0},\mathcal{I}})=\LDet_\theta(\B^{H,+}_{\bar{0},\mathcal{I}})-\LDet_\theta(-\B^{H,+}_{\bar{1},\mathcal{I}}).
\end{equation}
It follows from formula (6-17) of \cite{BK3} that \eqref{E:graddet} is independent of the choice of $\theta \in (-\pi, 0)$.

\subsection{The canonical element of the determinant line} It is not difficult to check that $(\nabla^H)^2=0$.
Clearly, $\nabla^H : \Omega^{\bar k}(M,E) \to \Omega^{\overline{k+1}}(M,E)$ and $\Gamma : \Omega^{\bar k}(M,E) \to \Omega^{\overline{k+1}}(M,E)$. Hence we can consider the following twisted de Rham complex with chirality operator $\Gamma$:
\begin{equation}\label{E:ztcomplex}
\big(\, \Omega^{  \bullet}(M,E), \nabla^H \, \big): \cdots \stackrel{\nabla^H}{\longrightarrow} \Omega^{\bar 0}(M,E ) \stackrel{\nabla^H}{\longrightarrow} \Omega^{\bar 1}(M,E ) \stackrel{\nabla^H}{\longrightarrow} \Omega^{\bar 0}(M,E )\stackrel{\nabla^H}{\longrightarrow} \cdots.
\end{equation}
We define the {\em twisted de Rham cohomology groups of $(\Omega^{\bullet}(M,E),\nabla^H)$} as
\[
H^{\bar{k}}(M,E,H) \, \equiv  \, H^{\bar{k}}(\nabla^H) \, := \, \frac{\Ker(\nabla^H : \Omega^{\bar{k}}(M,E) \to \Omega^{\overline{k+1}}(M,E))}{\im (\nabla^H : \Omega^{\overline{k+1}}(M,E) \to \Omega^{\bar{k}}(M,E))}, \quad k=0,1.
\] 
The groups $H^{\bar{k}}(M,E,H),k=0,1$ are independent of the choice of the Riemannian metric on $M$ or the Hermitian metric on $E$. Suppose that $H$ is repalced by $H'=H-dB$ for some $B \in \Omega^{\bar{0}}(M)$, there is an isomorphism $\varepsilon_B:=e^B \wedge \cdot : \Omega^{  \bullet}(M,E) \to \Omega^{  \bullet}(M,E)$ satisfying 
\[
\varepsilon_B \circ \nabla^H = \nabla^{H'} \circ \varepsilon_B.
\]
Therefore $\varepsilon_B$ induces an isomorphism on the twisted de Rham cohomology, also denote by $\varepsilon_B$,
\begin{equation}\label{E:epsiso}
\varepsilon_B: H^\bullet(M,E,H) \to H^\bullet(M,E,H').
\end{equation}

Denote by $(\nabla^{H}_{\bar k})^*$ the adjoint of $\nabla^H_{\bar k}$ with respect to the scalar product $<\cdot,\cdot>_M$. Then the Laplacians
\[
\Delta_{\bar k}=\Delta^H_{\bar k}:=(\nabla_{\bar k}^{H})^*\nabla_{\bar k}^H + \nabla_{\overline{k+1}}^H(\nabla_{\overline{k+1}}^{H})^*, \quad k=0,1
\]
are elliptic operators and therefore the complex \eqref{E:ztcomplex} is elliptic. By Hodge theory, we have the isomorphism $\Ker \Delta_{\bar k} \cong H^{\bar{k}}(M,E,H),k=0,1$. For more details of the twisted de Rham cohomology, cf. for example \cite{MW}.

Since $\nabla^H$ commutes with $\B^H$, the subspace $\Omega^{  \bullet}_{\mathcal{I}}(M,E)$ is a subcomplex of the twisted de Rham complex $(\Omega^{\bullet}(M,E), \nabla^H)$. Clearly, for each $\lambda \ge 0$, the complex $\Omega^{  \bullet}_{(\lambda, \infty)}(M,E)$ is acyclic. Since
\begin{equation}\label{E:omespl}
\Omega^{  \bullet}(M,E) = \Omega^{ \bullet}_{[0,\lambda]}(M,E) \oplus \Omega^{ \bullet}_{(\lambda,\infty)}(M,E),
\end{equation}
the cohomology $H^{  \bullet}_{[0,\lambda]}(\nabla^H) \equiv H^{  \bullet}_{[0,\lambda]}(M,E,H)$ of the complex $\big( \Omega^{  \bullet}(M,E), \nabla^H \big)$ is naturally isomorphic to the cohomology $H^{\bullet}(M,E,H)$. Let $\Gamma_{\mathcal{I}}$ denote the restriction of $\Gamma$ to $\Omega^{  \bullet}_{\mathcal{I}}(M,E)$. For each $\lambda \ge 0$, let
\begin{equation}\label{E:canonele}
\rho_{\Gamma_{[0,\lambda]}} = \rho_{\Gamma_{[0,\lambda]}}(\nabla^H,g^M) \in \Det\big(H^{  \bullet}_{[0,\lambda]}(M,E,H)\big)
\end{equation}
denote the refined torsion of the twisted finite dimensional complex $\big( \Omega^{  \bullet}_{[0,\lambda]}(M,E), \nabla^H \big)$ corresponding to the chirality operator $\Gamma_{[0,\lambda]}$, cf. Definition~\ref{D:elerho1}.
We view $\rho_{\Gamma_{[0,\lambda]}}$ as an element of $\Det\big(H^\bullet(M,E,H)\big)$ via the canonical isomorphism between $H^\bullet(M,E,H)$ and $H^{\bullet}_{[0,\lambda]}(M,E,H)$.

\begin{proposition}
Assume that $\theta \in (-\pi,0)$ is an Agmon angle for the operator $\B^H_{\bar 0}$. Then the element 
\begin{equation}\label{E:rhoh}
\rho_{H}=\rho(\nabla^H,g^M):= \Det_{\gr,\theta}(\B^{H}_{\bar{0},(\lambda,\infty)}) \cdot \rho_{\Gamma_{[0,\lambda]}} \in \Det\big(H^\bullet(M,E,H)\big)
\end{equation}
is independent of the choice $\lambda \ge 0$. Further, $\rho_H$ is independent of the choice of the Agmon angle $\theta \in (-\pi,0)$ of $\B^H_{\bar 0}$.
\end{proposition}
\begin{proof}
Clearly, for $0 \le \lambda \le \mu$, we have
\begin{equation}\label{E:indlam}
\Det_{\operatorname{gr}}(\B_{\bar{0},(\lambda, \infty)}^{H})=\Det_{\operatorname{gr}}(\B_{\bar{0},(\lambda, \mu]}^{H}) \cdot \Det_{\operatorname{gr}}(\B_{\bar{0},(\mu, \infty)}^{H}).
\end{equation}
From Proposition~\ref{P:indedec}, \eqref{E:indlam} and (6-17) of \cite{BK3}, we obtain the result.
\end{proof}

\subsection{The $\eta$-invariant}
In this subsection we recall the definition of the $\eta$-invariant of a non-self-adjoint elliptic operator $D$, cf. \cite{Gi}, \cite[Subsection 6.15]{BK3}. 

\begin{definition}
Let $D:C^\infty(M,E) \to C^\infty(M,E)$ be an elliptic differential operator of order $n \ge 1$ whose leading symbol is self-adjoint with respect to some given Hermitian metric on $E$. Assume that $\theta$ is an Agmon angle for $D$, cf. Definition 6.3 of \cite{BK3}. Let $\Pi_>$ (resp. $\Pi<$) be a pseudo-differential projection whose image contains the span of all generalized eigenvectors of $D$ corresponding to eigenvalues $\lambda$ with $\operatorname{Re} >0$ (resp. with $\operatorname{Re}  <0$) and whose kernel contains the span of all generalized eigenvectors of $D$ corresponding to eigenvalues $\lambda$ with $\operatorname{Re} \le 0$ (resp. with $\operatorname{Re} \ge 0$). We define the $\eta$-function of $D$ by the formula
\[
\eta_\theta(s,D)=\zeta_\theta(s, \Pi_>,D)- \zeta_\theta(s, \Pi_<,-D).
\]
\end{definition} Note that, by the above definition, the purely imaginary eigenvalues of $D$ do not contribute to $\eta_\theta(s,D)$.

It was shown by Gilkey, \cite{Gi}, that $\eta_\theta(s,D)$ has a meromorphic extension to the whole complex plane $\mathbb{C}$ with isolated simple poles, and that it is regular at $0$. Moreover, the number $\eta_\theta(0,D)$ is independent of the Agmon angle $\theta$.

Since the leading symbol of $D$ is self-adjoint, the angles $\pm \pi/2$ are principal angles for $D$ cf. \cite[Definition 6.2]{BK3}. Hence, there are at most finitely many eigenvalues of $D$ on the imaginary axis. Let $m_+(D)$ (resp. $m_-(D)$) denote the number of eigenvalues of $D$, counted with their algebraic multiplicities, on the positive (resp. negative) part of the imaginary axis. Let $m_0(D)$ denote the algebraic multiplicity of $0$ as an eigenvalue of $D$.

\begin{definition}\label{D:etainv}
The $\eta$-invariant $\eta(D)$ of $D$ is defined by the formula
\begin{equation}\label{E:etainv11}
\eta(D)=\frac{\eta_\theta(0,D)+m_+(D)-m_-(D)+m_0(D)}{2}.
\end{equation}
\end{definition}
Since $\eta_\theta(0,D)$ is indenpendent of the choice of the Agmon angle $\theta$ for $D$, cf. \cite{Gi}, so is $\eta(D)$. Note that the defintion of $\eta(D)$ is slightly different from the one proposed by Gilkey in \cite{Gi}. See \cite[Remark 2.5]{BK3}. 

Denote by $\eta(\nabla^H)=\eta(\B^H_{\bar 0})$ the $\eta$-invariant of the restriction $\B^H_{\bar 0}$ of the twisted odd signature operator $\B^H$ to $\Omega^{\bar 0}(M,E)$.

\subsection{Relationship with the $\eta$-invariant}
In this subsection we study the relationship between \eqref{E:graddet} and the $\eta$-invariant of $\B_{\bar{0},(\lambda, \infty)}^{H}$.

To simplify the notation set
\begin{equation}\label{E:etano}
\eta_\lambda(\nabla^H):=\eta(\B_{\bar{0},(\lambda, \infty)}^{H})
\end{equation}
and
\begin{equation}\label{E:xino}
\xi_\lambda = \xi_\lambda(\nabla^H,g^M,\theta) = \frac{1}{2}\Big(\, \LDet_{2\theta}\big((\B_{\bar{0},(\lambda, \infty)}^{H,+})^2\big) - \LDet_{2\theta}\big((\B_{\bar{1},(\lambda, \infty)}^{H,+})^2\big) \, \Big).
\end{equation}

Let $P^\pm_{\bar{k},\mathcal{I}}, k=0,1$ be the orthogonal projection onto the closure of the subspace $\Omega^{\bar k}_{\pm,\mathcal{I}}(M,E)$. Set
\begin{equation}\label{E:dno}
d_{\bar{k},\lambda}^{\mp}:= \operatorname{rank}(\operatorname{Id}-P^{\pm}_{\bar{k},[0,\lambda]}) = \dim \Omega^{\bar k}_{\mp,\mathcal{I}}(M,E), \quad k=0,1.
\end{equation}

If $\mathcal{I} \subset \mathbb{R}$ we denote by $L_{\mathcal{I}}$ the solid angle
\[
L_{\mathcal{I}}=\{\rho e^{i \theta}: 0 < \rho < \infty, \theta \in \mathcal{I} \}.
\]

\begin{proposition}
Let $\nabla$ be a flat connection on a vector bundle $E$ over a closed Riemannian manifold $(M,g^M)$ of odd dimension $m=2r-1$ and $H$ is an odd-degree closed differential form, other than one form, on $M$. Assume $\theta \in (-\pi/2,0)$ is an Agmon angle for the twisted odd signature operator $\B_{\bar{0},(\lambda, \infty)}^{H}$ such that there are no eigenvalues of $\B^H$ in the solid angles $L_{(-\pi/2,\theta]}$ and $L_{(\pi/2,\theta + \pi]}$. Then, for every $\lambda \ge 0$, cf. \eqref{E:ldetgr},
\begin{equation}\label{E:ldetetarell}
\LDet_{\gr,\theta}(\B^{H}_{\bar{0},(\lambda,\infty)})=\xi_\lambda - i \pi \eta_\lambda(\nabla^H) -\frac{i \pi}{2}\sum_{k=0,1}(-1)^k d^-_{\bar{k},\lambda}.
\end{equation} 
\end{proposition}
\begin{proof}
From Definition \ref{D:etainv} of the $\eta$-invariant it follows that 
\begin{equation}\label{E:eteq1}
\eta(-\B_{\bar{1},(\lambda, \infty)}^{H,+}) = - \eta(\B_{\bar{1},(\lambda, \infty)}^{H,+}).
\end{equation}
By the fact that $\Gamma \B_{\bar{1},(\lambda, \infty)}^{H,+} \Gamma = \B_{\bar{0},(\lambda, \infty)}^{H,-}$,
we have
\begin{equation}\label{E:eteq2}
\eta(\B_{\bar{1},(\lambda, \infty)}^{H,+}) = \eta(\B_{\bar{0},(\lambda, \infty)}^{H,-}).
\end{equation}
Combining \eqref{E:etano}, \eqref{E:eteq1} with \eqref{E:eteq2}, we have
\begin{equation}\label{E:etarelat}
\eta(\B_{\bar{0},(\lambda, \infty)}^{H,+})-\eta(-\B_{\bar{1},(\lambda, \infty)}^{H,+})=\eta(\B_{\bar{0},(\lambda, \infty)}^{H,+})+\eta(\B_{\bar{0},(\lambda, \infty)}^{H,-})=\eta_\lambda(\nabla^H).
\end{equation}
By \cite[(4.34)]{BK1}, for $k=0,1$, we have
\begin{equation}\label{E:ldetpmrel}
\LDet_\theta(\pm \B^{H,+}_{\bar{k},\mathcal{I}}) \, = \, \frac{1}{2}\LDet_{2\theta}\big((\B_{\bar{k},(\lambda, \infty)}^{H,+})^2\big)- i \pi \Big(\, \eta(\pm \B^{H,+}_{\bar{k},\mathcal{I}})-\frac{\zeta_{2\theta}(0, \big(\B_{\bar{k},(\lambda, \infty)}^{H,+}\big)^2)}{2} \,\Big)
\end{equation}
and, by \cite[(6-6)]{BK3} and \eqref{E:dno}, we have
\begin{equation}\label{E:zettw}
\zeta_{2\theta}(0, \big(\B_{\bar{0},(\lambda, \infty)}^{H,+}\big)^2)-\zeta_{2\theta}(0, \big(\B_{\bar{1},(\lambda, \infty)}^{H,+}\big)^2)= -\sum_{k=0,1}(-1)^k d^-_{\bar{k},\lambda}.
\end{equation}
Combining \eqref{E:xino}, \eqref{E:etarelat}, \eqref{E:ldetpmrel} with \eqref{E:zettw}, we obtain the result.
\end{proof}

\section{Metric anomaly and the definition of the refined analytic torsion twisted by a flux form}\label{S:madrat}

In this section we study the metric dependence of the element $\rho_H=\rho(\nabla^H,g^M)$ defined in \eqref{E:rhoh}. We then use this element to construct the {\em twisted refined analytic torsion}, which is a canonical element of the determinant line $\Det(H^\bullet(M,E,H))$. We also show that the twisted refined analytic torsion is independent of the metric $g^M$ and the representative $H$ in the cohomology class $[H]$.
  
\subsection{Relationship between $\rho_H(t)$ and the $\eta$-invariant}
Suppose that $g^M_t, t \in \mathbb{R}$, is a smooth family of Riemannian metrics on $M$. Let
\[
\rho_H(t)=\rho(\nabla^H,g^M_t) \in \Det\big( H^\bullet(M,E,H) \big)
\]
be the canonical element defined in \eqref{E:rhoh}. 

Let $\Gamma_t$ denote the chirality operator corresponding to the metric $g^M_t$, cf. \eqref{E:chirality}, and let $\B^H(t)=\B(\nabla^H,g_t^M)$ denote the twisted odd signature operator corresponding to the Riemannian metric $g^M_t$ and let $\B^{H,+}(t)$ denote the restriction of $\B^H(t)=\B(\nabla^H,g_t^M)$ to $\Omega^{\bullet}_+(M,E)$.

Fix $t_0 \in \mathbb{R}$ and choose $\lambda \ge 0$ such that there are no eigenvalues of $(\B^{H}(t_0))^2$ of absolute value $\lambda$. Further, assume that $\lambda$ is big enough so that the real parts of eigenvalues of $(\B^{H}_{(\lambda,\infty)}(t_0))^2$ are all greater than $0$. Then there exists $\delta > 0$ small enough such that the same holds for the spectrum of $(\B^{H}(t))^2$ for $t \in (t_0-\delta,t_0+\delta)$. In particular, $d^-_{\bar{k},\lambda}$, cf. \eqref{E:dno}, is independent of $t \in (t_0-\delta,t_0+\delta)$. Set
\[
\eta_\lambda(\nabla^H,t):=\eta(\B_{\bar{0},(\lambda, \infty)}^{H}(t)), \quad
\xi_\lambda(t,\theta) = \xi_\lambda(\nabla^H,g_t^M,\theta).
\]
By definition \eqref{E:rhoh},
\[
\rho_H(t)= \Det_{\gr,\theta}(\B^{H}_{\bar{0},(\lambda,\infty)}(t)) \cdot \rho_{\Gamma_{t,[0,\lambda]}}.
\]
Assume that $\theta_0 \in (\pi/2,0)$ is an Agmon angle for $\B^{H}(t_0)$ such that there are no eigenvalues of $\B^{H}(t_0)$ in $L_{(-\pi/2,\theta_0]}$ and $L_{(-\pi/2,\theta_0 + \pi)}$. Choose $\delta > 0$ so that for every $t \in (t_0-\delta,t_0+\delta)$ both $\theta_0$ and $\theta_0+\pi$ are Agmon angles of $\B^{H}(t_0)$. For $t \not= t_0$, it might happen that there are eigenvalues of $\B_{\bar{k},(\lambda, \infty)}^{H}(t)$ in $L_{(-\pi/2,\theta_0]}$ and $L_{(-\pi/2,\theta_0 + \pi)}$. Hence, \eqref{E:ldetetarell} is not necessarily true. However, from the independence of the Agmon angle of the $\zeta$-function, cf. \cite[(6-16)]{BK3}, and \eqref{E:xino}, we conclude that for every angle $\theta \in (-\pi/2,0)$, so that $\theta$ and $\theta + \pi$ are Agmon angles for $\B_{\bar{k},(\lambda, \infty)}^{H}(t)$,
\[
\xi_{\lambda}(t,\theta) \equiv \xi_{\lambda}(t,\theta_0) \quad \text{mod} \, \pi i. 
\] 
Hence, from \eqref{E:ldetetarell}, we obtain
\begin{equation}\label{E:pmfor}
\rho_H(t)= \pm e^{\xi_{\lambda}(t,\theta_0)} \cdot e^{-i \pi \eta_\lambda(\nabla^H,t)} \cdot e^{\frac{i \pi}{2}\sum_{k=0,1}(-1)^k d^-_{\bar{k},\lambda}} \cdot  \rho_{\Gamma_{t,[0,\lambda]}}.
\end{equation}

\begin{lemma}\label{L:meind}
Under the above assumptions, the product
\[
e^{\xi_{\lambda}(t,\theta_0)} \cdot \rho_{\Gamma_{t,[0,\lambda]}} \in \Det\big( H^{\bullet}\big(M,E,H\big) \big)
\]
is independent of $t \in (t_0-\delta, t_0+\delta)$.
\end{lemma}
\begin{proof}
Let $\Gamma_t$ denote the chirality operator corresponding to the metric $g_t^M$. Following the $\mathbb{Z}$-graded case, cf. \cite{RS}, we set  
\begin{align}\label{E:plva3}
f(s,t)& =  \sum_{k=0,1}(-1)^k \int_0^\infty u^{s-1} \Tr \Big[\exp \big(-u(\Gamma_t \nabla^H)^2 \big|_{\Omega^{\bar k}_{+,(\lambda,\infty)}(M,E)}\big) \Big]du \nonumber \\
& = \Gamma(s) \sum_{k=0,1}(-1)^k  \zeta \Big(s, (\Gamma_t \nabla^H)^2 \big|_{\Omega^{\bar k}_{+,(\lambda,\infty)}(M,E)} \Big)
\end{align}
Using the fact that 
$$
\Gamma_t (\Gamma_t \nabla^H)^2 \big|_{\Omega^{\bar k}_{+,(\lambda,\infty)}(M,E)} \Gamma_t = (\nabla^H \Gamma_t )^2 \big|_{\Omega^{\overline{k+1}}_{-,(\lambda,\infty)}(M,E)},
$$ we also have
\begin{align}\label{E:plva4}
f(s,t)& =  -\sum_{k=0,1}(-1)^k \int_0^\infty u^{s-1} \Tr \Big[\exp \big(-u(\nabla^H\Gamma_t )^2 \big|_{\Omega^{\bar k}_{-,(\lambda,\infty)}(M,E)}\big) \Big]du \nonumber \\
& =  -\Gamma(s) \sum_{k=0,1}(-1)^k  \zeta \Big(s, (\nabla^H\Gamma_t )^2 \big|_{\Omega^{\bar k}_{-,(\lambda,\infty)}(M,E)} \Big)
\end{align}

We denote by $\dot{\Gamma}_t$ with respect to the parameter $t$. Then 
\begin{align}\label{E:plva1}
\frac{d}{dt} (\Gamma_t \nabla^H)^2 \big|_{\Omega^{\bar k}_{+,(\lambda,\infty)}(M,E)} & =  \dot{\Gamma}_t \Gamma_t (\Gamma_t \nabla^H)^2 \big|_{\Omega^{\bar k}_{+,(\lambda,\infty)}(M,E)} \nonumber \\
 & + (\Gamma_t \nabla^H) \big|_{\Omega^{\bar k}_{+,(\lambda,\infty)}(M,E)} \dot{\Gamma}_t \Gamma_t (\Gamma_t \nabla^H) \big|_{\Omega^{\bar k}_{+,(\lambda,\infty)}(M,E)}
\end{align}
where we used that $\Gamma^2_t=1$. 
Similarly, we have
\begin{equation}\label{E:plva2}
\frac{d}{dt} ( \nabla^H \Gamma_t)^2 \big|_{\Omega^{\bar k}_{-,(\lambda,\infty)}(M,E)} =   ( \nabla^H \Gamma_t)^2\Gamma_t \dot{\Gamma}_t \big|_{\Omega^{\bar k}_{-,(\lambda,\infty)}(M,E)}+ ( \nabla^H \Gamma_t)\Gamma_t \dot{\Gamma}_t( \nabla^H \Gamma_t)\big|_{\Omega^{\bar k}_{-,(\lambda,\infty)}(M,E)}.
\end{equation}
If $A$ is of trace class and $B$ is a bounded operator, it is well known that $\operatorname{Tr}(AB)=\operatorname{Tr}(BA)$. Hence, by this fact and the semi-group property of the heat operator, we have
\begin{align}\label{E:trequality1}
  & \Tr \Big[ (\Gamma_t \nabla^H) \big|_{\Omega^{\bar k}_{+,(\lambda,\infty)}(M,E)} \dot{\Gamma}_t \Gamma_t (\Gamma_t \nabla^H) \big|_{\Omega^{\bar k}_{+,(\lambda,\infty)}(M,E)} \exp \big(-u(\Gamma_t \nabla^H)^2 \big|_{\Omega^{\bar k}_{+,(\lambda,\infty)}(M,E)}\big) \Big] \nonumber \\
 = & \Tr \Big[ \exp \big(-\frac{u}{2}(\Gamma_t \nabla^H)^2 \big|_{\Omega^{\bar k}_{+,(\lambda,\infty)}(M,E)}\big) (\Gamma_t \nabla^H) \big|_{\Omega^{\bar k}_{+,(\lambda,\infty)}(M,E)} \nonumber \\
 & \cdot   \dot{\Gamma}_t \Gamma_t (\Gamma_t \nabla^H) \big|_{\Omega^{\bar k}_{+,(\lambda,\infty)}(M,E)} \exp \big(-\frac{u}{2}(\Gamma_t \nabla^H)^2 \big|_{\Omega^{\bar k}_{+,(\lambda,\infty)}(M,E)}\big) \Big]\nonumber \\
= & \Tr \Big[\dot{\Gamma}_t \Gamma_t (\Gamma_t \nabla^H) \big|_{\Omega^{\bar k}_{+,(\lambda,\infty)}(M,E)} \exp \big(-\frac{u}{2}(\Gamma_t \nabla^H)^2 \big|_{\Omega^{\bar k}_{+,(\lambda,\infty)}(M,E)}\big) \nonumber \\
 & \cdot \exp \big(-\frac{u}{2}(\Gamma_t \nabla^H)^2 \big|_{\Omega^{\bar k}_{+,(\lambda,\infty)}(M,E)}\big) (\Gamma_t \nabla^H) \big|_{\Omega^{\bar k}_{+,(\lambda,\infty)}(M,E)}   \Big] \nonumber \\
 = & \Tr \Big[ \dot{\Gamma}_t \Gamma_t (\Gamma_t \nabla^H) \big|_{\Omega^{\bar k}_{+,(\lambda,\infty)}(M,E)} \exp \big(-u(\Gamma_t \nabla^H)^2 \big|_{\Omega^{\bar k}_{+,(\lambda,\infty)}(M,E)}\big) (\Gamma_t \nabla^H) \big|_{\Omega^{\bar k}_{+,(\lambda,\infty)}(M,E)}   \Big] \nonumber \\
 = & \Tr \Big[ \dot{\Gamma}_t \Gamma_t (\Gamma_t \nabla^H)^2 \big|_{\Omega^{\bar k}_{+,(\lambda,\infty)}(M,E)}\exp \big(-u(\Gamma_t \nabla^H)^2 \big|_{\Omega^{\bar k}_{+,(\lambda,\infty)}(M,E)}\big) \Big],
\end{align}
here in the last equality we used the fact that
\begin{align}
& (\Gamma_t \nabla^H) \big|_{\Omega^{\bar k}_{+,(\lambda,\infty)}(M,E)} \exp \big(-u(\Gamma_t \nabla^H)^2 \big|_{\Omega^{\bar k}_{+,(\lambda,\infty)}(M,E)}\big) (\Gamma_t \nabla^H) \big|_{\Omega^{\bar k}_{+,(\lambda,\infty)}(M,E)} \nonumber \\
= &(\Gamma_t \nabla^H)^2 \big|_{\Omega^{\bar k}_{+,(\lambda,\infty)}(M,E)}\exp \big(-u(\Gamma_t \nabla^H)^2 \big|_{\Omega^{\bar k}_{+,(\lambda,\infty)}(M,E)}\big). \nonumber
\end{align}
By \eqref{E:plva1}, \eqref{E:plva3} and \eqref{E:trequality1}, we have
\begin{equation}\label{E:plva5}
\frac{d}{dt}f(s,t) = \sum_{k=0,1}(-1)^k \int_0^\infty u^{s-1} \Tr \Big[- 2u \dot{\Gamma}_t \Gamma_t (\Gamma_t \nabla^H)^2 \big|_{\Omega^{\bar k}_{+,(\lambda,\infty)}(M,E)}\exp \big(-u(\Gamma_t \nabla^H)^2 \big|_{\Omega^{\bar k}_{+,(\lambda,\infty)}(M,E)}\big) \Big]du.
\end{equation}
Simlarly, we have
\begin{equation}\label{E:plva6}
\frac{d}{dt}f(s,t) = -\sum_{k=0,1}(-1)^k \int_0^\infty u^{s-1} \Tr \Big[- 2u \Gamma_t \dot{\Gamma}_t  (\nabla^H \Gamma_t )^2 \big|_{\Omega^{\bar k}_{-,(\lambda,\infty)}(M,E)}\exp \big(-u(\nabla^H \Gamma_t )^2 \big|_{\Omega^{\bar k}_{-,(\lambda,\infty)}(M,E)}\big) \Big]du.
\end{equation}By \eqref{E:plva5}, \eqref{E:plva6} and the fact that $\dot{\Gamma}_t \Gamma_t = - \Gamma_t \dot{\Gamma}_t$, we conclude that
\begin{align}\label{E:plva7}
\frac{d}{dt}f(s,t) & =  - \sum_{k=0,1}(-1)^k \int_0^\infty u^{s} \Tr \Big[ \dot{\Gamma}_t \Gamma_t (\B_{\bar{k},(\lambda,\infty)}^{H}(t))^2 \exp \big(-u(\B_{\bar{k},(\lambda,\infty)}^{H}(t))^2 \big)  \Big]du \nonumber \\
& =  \sum_{k=0,1}(-1)^k \int_0^\infty u^{s} \frac{d}{du} \Tr \Big[ \dot{\Gamma}_t \Gamma_t \exp \big(-u(\B_{\bar{k},(\lambda,\infty)}^{H}(t))^2 \big)  \Big]du \nonumber  \\
& = - s\sum_{k=0,1}(-1)^k \int_0^\infty u^{s-1}  \Tr \Big[ \dot{\Gamma}_t \Gamma_t \exp \big(-u(\B_{\bar{k},(\lambda,\infty)}^{H}(t))^2 \big)  \Big]du,
\end{align}
where we used the integration by parts for the last equality.
Since $(\B^H(t))^2$ is an elliptic differential operator, the dimension of $\Omega^\bullet_{[0,\lambda]}(M,E)$ is finite. Let $\varepsilon \not=0$ be a small enough real number so that $(\B^H(t))^2 + \varepsilon$ is bijective and $2\theta_0$ is an Agmon angle for $(\B^H(t))^2 + \varepsilon$. Then we can rewrite \eqref{E:plva7} as
\begin{align}\label{E:plva8}
\frac{d}{dt}f(s,t) & =  - s\sum_{k=0,1}(-1)^k \int_0^1 u^{s-1}  \Tr \Big[ \dot{\Gamma}_t \Gamma_t \big(\exp \big(-u(\B_{\bar{k}}^{H}(t))^2+\varepsilon  \big) \Big]du \nonumber \\ 
&   - s\sum_{k=0,1}(-1)^k \int_1^\infty u^{s-1}  \Tr \Big[ \dot{\Gamma}_t \Gamma_t \big(\exp \big(-u(\B_{\bar k}^{H}(t))^2+\varepsilon  \big) \Big]du \nonumber \\ 
&   + s\sum_{k=0,1}(-1)^k \int_0^1 u^{s-1}  \Tr \Big[ \dot{\Gamma}_t \Gamma_t \big(\exp \big(- u (\B_{\bar{k},[0,\lambda]}^{H}(t))^2 \big) \Big]du \nonumber \\
&   + s \sum_{k=0,1}(-1)^k \int_1^\infty u^{s-1}  \Tr \Big[ \dot{\Gamma}_t \Gamma_t \big(\exp \big(- u (\B_{\bar{k},[0,\lambda]}^{H}(t))^2 \big) \Big]du.
\end{align}
Since $\dot{\Gamma}_t \Gamma_t$ is a local quantitity and the dimension of the manifold $M$ is odd, the asymptotic expansion as $u \downarrow 0$ for $\Tr \Big[ \dot{\Gamma}_t \Gamma_t \big(\exp \big(-u(\B^{H}(t))^2+\varepsilon  \big) \Big]$ does not contain a constant term. Therefore the integrals of the first term on the right hand side of \eqref{E:plva8} do not have poles at $s=0$. On the other hand, because of exponential decay of $\Tr \Big[ \dot{\Gamma}_t \Gamma_t \big(\exp \big(-u(\B^{H}(t))^2+\varepsilon  \big) \Big]$ and $\Tr \Big[ \dot{\Gamma}_t \Gamma_t \big(\exp \big(- u (\B_{[0,\lambda]}^{H}(t))^2 \big) \Big]$ for large $u$, the integrals of the second term and the fourth term on the right hand side of \eqref{E:plva8} are entire functions in $s$. Hence we have
\begin{align}\label{E:plva9}
\frac{d}{dt}\Big|_{s=0}f(s,t) & =  \Big(s \sum_{k=0,1}(-1)^k \int_0^1 u^{s-1} \Tr \Big[ \dot{\Gamma}_t \Gamma_t \Big|_{\Omega^{\bar k}_{[0,\lambda]}(M,E)} \Big] du \Big) \Big|_{s=0} \nonumber \\
& =  \sum_{k=0,1}(-1)^k \Tr \Big[ \dot{\Gamma}_t \Gamma_t \Big|_{\Omega^{\bar k}_{[0,\lambda]}(M,E)} \Big]
\end{align}
and, by \eqref{E:plva3},
\begin{equation}\label{E:plva10}
\frac{d}{dt}\Big|_{s=0}sf(s,t) = \frac{d}{dt}\Big|_{s=0} \sum_{k=0,1}(-1)^k \zeta \Big(s, (\Gamma_t \nabla^H)^2 \big|_{\Omega^{\bar k}_{+,(\lambda,\infty)}(M,E)} \Big) =0.
\end{equation}
By \eqref{E:xino} and \eqref{E:plva3}, we know that 
\begin{equation}\label{E:plva11}
\xi_{\lambda}(t,\theta_0)= -\frac{1}{2}\operatorname{lim}_{s \to 0}\Big[ f(s,t)-\frac{1}{s} \sum_{k=0,1}(-1) \zeta(0, (\Gamma_t \nabla^H)^2 \big|_{\Omega^{\bar k}_{+,(\lambda,\infty)}(M,E)})  \Big].
\end{equation}
Combining \eqref{E:plva9}, \eqref{E:plva10} with \eqref{E:plva11}, we obtain
\begin{equation}\label{E:xivarfor}
\frac{d}{dt}\Big|_{t=0}\xi_{\lambda}(t,\theta_0)= -\frac{1}{2}\sum_{k=0,1}(-1)^k \Tr \Big[ \dot{\Gamma}_t \Gamma_t \Big|_{\Omega^{\bar k}_{[0,\lambda]}(M,E)} \Big]
\end{equation}
Combining \eqref{E:varfor} with \eqref{E:xivarfor}, we obtain
\[
\frac{d}{dt}(e^{\xi_{\lambda}(t,\theta_0)} \cdot \rho_{\Gamma_{t,[0,\lambda]}})=0.
\]
\end{proof}

We need the following lemma, which is a slight modification of the result in Subsection 9.3 of \cite{BK3}.
\begin{lemma}\label{L:etaf}
For any $t_1,t_2 \in (t_0-\delta, t_0 + \delta)$, we have
\[
\eta_\lambda(\nabla^H,t_1)-\eta_\lambda(\nabla^H,t_2) \equiv \eta(\B^{H}_{\bar{0}}(t_1)) - \eta(\B^{H}_{\bar{0}}(t_2)), \quad \operatorname{mod} \ \mathbb{Z}. 
\]
\end{lemma}
Let $\B^H_{\trivial}=\B_{\bar 0}(\nabla^H_{\trivial},g^M):\Omega^{\bar 0}(M,E) \to \Omega^{\bar 0}(M,E)$ denote the even part of twisted odd signature operator corresponding to the metric $g^M$ and the trivial line bundle over $M$ endowed with the trivial connection $\nabla_{\trivial}$. Put
\[
\eta_{\trivial}:=\frac{1}{2} \eta(0,\B_{\trivial}^{H}).
\]

We now need to study the dependence of $\eta(\B^{H}_{\bar{0}})$ on the Riemannian metric $g^M$. This was essentially done in \cite{APS2} and \cite{Gi}. 
\begin{lemma}\label{L:rhopm}
The function $\eta(\B^{H}_{\bar{0}}(t))-\operatorname{rank}(E)\eta_{\trivial}(t)$ is, modulo $\mathbb{Z}$, independent of $t \in (t_0-\delta, t_0+\delta)$.
\end{lemma}

\subsection{Removing the metric anomaly and the definition of the twisted refined analytic torsion}
The following theorem is the main theorem of this subsection.
\begin{theorem}
Let $M$ be an odd dimensional oriented closed Riemannian manifold. Let $(E,\nabla,h^E)$ be a flat complex vector bundle over $M$ and $H$ is a closed differential form on $M$ of odd degree, other than one form. Then the element 
\begin{equation}\label{E:reffor}
\rho(\nabla^H,g^M) \cdot e^{i \pi (\operatorname{rank} (E))\eta_{\trivial}} \in \Det\big(H^\bullet(M,E,H)\big),
\end{equation}
where $\nabla^H:=\nabla + H \wedge \cdot$ and $\rho(\nabla^H,g^M) \in \Det\big(H^\bullet(M,E,H)\big)$ is defined in \eqref{E:rhoh}, is independent of $g^M$.
\end{theorem}
\begin{proof}
Consider a smooth family $g^M_t, t \in \mathbb{R}$ of Riemannian metrics on $M$. From \eqref{E:pmfor}, we obtain for $t \in (t_0-\delta, t_0+\delta)$
\begin{equation}\label{E:meana}
\rho^H(t) \cdot e^{i \pi (\operatorname{rank} (E))\eta_{\trivial}}=\pm e^{\xi_{\lambda}(t,\theta_0)} \cdot e^{-i \pi \eta_\lambda(\nabla^H,t)} \cdot e^{\frac{i \pi}{2}\sum_{k=0,1}(-1)^kd_{\bar{k},\lambda}} \cdot  \rho_{\Gamma_{t,[0,\lambda]}} \cdot e^{i \pi (\operatorname{rank} (E))\eta_{\trivial}}.
\end{equation}
Combining \eqref{E:xivarfor}, \eqref{E:meana} with Lemma \ref{L:rhopm} we conclude that for any $t_1,t_2 \in (t_0-\delta,t_0+\delta)$
\[
\rho^H(t_1) \cdot e^{i \pi (\operatorname{rank} (E))\eta_{\trivial}}= \pm \rho^H(t_2) \cdot e^{i \pi (\operatorname{rank} (E))\eta_{\trivial}}.
\] 
Since the function $\rho^H(t) \cdot e^{i \pi (\operatorname{rank} (E))\eta_{\trivial}}$ is continuous and nonzero, the sign in the right hand side of the equality must be positive. This proves the statement.
\end{proof}

\begin{definition}\label{D:maindef}
Let $M$ be an odd dimensional oriented closed Riemannian manifold. Let $(E,\nabla,h^E)$ be a flat complex vector bundle over $M$ and $H$ is a closed differential form on $M$ of odd degree, other than one form. The {\em twisted refined analytic torsion} $\rho_{\operatorname{an}}(\nabla^H)$ is the element of $\Det\big(H^\bullet(M,E,H)\big)$ defined by \eqref{E:reffor}.
\end{definition}

\subsection{Variation of refined analytic torsion with respect to the flux in a cohomology class}
Suppose that the (real) flux form $H$ is deformed smoothly along a one-parameter family with parameter $v \in \mathbb{R}$ in such a way that the cohomology class $[H] \in H^{\bar 1}(M, \mathbb{R})$ is fixed. Then $\frac{d}{dv}H=-dB$ for some form $B \in \Omega^{\bar 0}(M)$ that depends smoothly on $v$. Let $\beta=B \wedge \cdot$.
Fix $v_0 \in \mathbb{R}$ and choose $\lambda > 0$ such that there are no eigenvalues of $(\B^H)^2(v_0)$ of absolute value $\lambda$. Further, assume that $\lambda$ is big enough so that the real parts of eigenvalues of $(\B^{H}_{(\lambda,\infty)}(v_0))^2$ are all greater than $0$. Then there exists $\delta > 0$ small enough such that the same holds for the spectrum of $(\B^{H}(v))^2$ for $v \in (v_0-\delta,v_0+\delta)$. For simplicity, we often omit the parameter $v$ in the notations of operators in the following discussion.

We have the following two lemmas, see also Lemma 3.5 and lemma 3.7 of \cite{MW}.
\begin{lemma}\label{L:flva1}
Under the above assumptions, we have
\[
\frac{d}{dv} \xi_\lambda = \sum_{k=0,1}(-1)^k \Tr (\beta \big|_{\Omega^{\bar k}_{[0,\lambda]}(M,E)}).
\]
\end{lemma}
\begin{proof}
As in the proof of Lemma~\ref{L:meind}, we set
\begin{equation}\label{E:fvh1}
f(s,v)  =  \sum_{k=0,1}(-1)^k \int_0^\infty u^{s-1} \Tr \Big[\exp \big(-u(\Gamma \nabla^H)^2 \big|_{\Omega^{\bar k}_{+,(\lambda,\infty)}(M,E)}\big) \Big]du.\\
\end{equation}
We note that $B$, hence $\beta$, is real. By \eqref{E:fvh1} and the fact that 
\[
\frac{d}{dv} \nabla^H = [\beta, \nabla^H],
\]
we have
\begin{align}\label{E:fvh2}
& \frac{d}{dv}f(s,v) \nonumber \\ & =  \sum_{k=0,1}(-1)^k \int_0^\infty u^{s-1} \Tr \Big[-u \big(\Gamma [\beta, \nabla^H] \Gamma \nabla^H \big)\big|_{\Omega^{\bar k}_{+,(\lambda,\infty)}(M,E)}\exp \big(-u(\Gamma \nabla^H)^2 \big|_{\Omega^{\bar k}_{+,(\lambda,\infty)}(M,E)}\big) \Big]du  \nonumber \\
 &+ \sum_{k=0,1}(-1)^k \int_0^\infty u^{s-1} \Tr \Big[-u  \big(\Gamma \nabla^H \Gamma [\beta, \nabla^H]\big)\big|_{\Omega^{\bar k}_{+,(\lambda,\infty)}(M,E)} \exp \big(-u(\Gamma \nabla^H)^2 \big|_{\Omega^{\bar k}_{+,(\lambda,\infty)}(M,E)}\big) \Big]du.
\end{align}
Using the fact that $\Gamma^2=1$, we have
\begin{align}\label{E:fvh3}
& \Tr \Big[ \big(\Gamma \beta \nabla^H \Gamma \nabla^H \big)\big|_{\Omega^{\bar k}_{+,(\lambda,\infty)}(M,E)}  \exp \big(-u(\Gamma \nabla^H)^2 \big|_{\Omega^{\bar k}_{+,(\lambda,\infty)}(M,E)}\big) \Big] \nonumber  \\
& = \Tr \Big[ \big(\Gamma \beta \Gamma \big)\big|_{\Omega^{\bar k}_{+,(\lambda,\infty)}(M,E)} \cdot \big( \Gamma \nabla^H \big)^2 \big|_{\Omega^{\bar k}_{+,(\lambda,\infty)}(M,E)}  \exp \big(-u(\Gamma \nabla^H)^2 \big|_{\Omega^{\bar k}_{+,(\lambda,\infty)}(M,E)}\big) \Big]. 
\end{align}
By using the trace property, $\Gamma^2=1$, and the semi-group property of the heat operator, we have
\begin{align}\label{E:fvh4}
& \Tr \Big[ \big(\Gamma  \nabla^H \beta \Gamma \nabla^H \big)\big|_{\Omega^{\bar k}_{+,(\lambda,\infty)}(M,E)} \exp \big( -u(\Gamma \nabla^H)^2 \big|_{\Omega^{\bar k}_{+,(\lambda,\infty)}(M,E)} \big) \Big] \nonumber  \\
& =  \Tr \Big[ \big(\Gamma  \nabla^H \Gamma \cdot \Gamma \beta \Gamma \nabla^H \big)\big|_{\Omega^{\bar k}_{+,(\lambda,\infty)}(M,E)}  \exp \big(-\frac{1}{2}u(\Gamma \nabla^H)^2 \big|_{\Omega^{\bar k}_{+,(\lambda,\infty)}(M,E)}\big) \nonumber\\ 
& \qquad \cdot \exp \big(-\frac{1}{2}u(\Gamma \nabla^H)^2 \big|_{\Omega^{\bar k}_{+,(\lambda,\infty)}(M,E)}\big) \Big]  \nonumber \\
&= \Tr \Big[ \exp \big(-\frac{1}{2}u(\Gamma \nabla^H)^2 \big|_{\Omega^{\bar k}_{+,(\lambda,\infty)}(M,E)}\big) \cdot \big(\Gamma  \nabla^H \Gamma \cdot \Gamma \beta \Gamma \nabla^H \big)\big|_{\Omega^{\bar k}_{+,(\lambda,\infty)}(M,E)} \nonumber \\ & \qquad  \cdot \exp \big(-\frac{1}{2}u(\Gamma \nabla^H)^2 \big|_{\Omega^{\bar k}_{+,(\lambda,\infty)}(M,E)}\big)  \Big] \nonumber  \\
& = \Tr \Big[ \big( \Gamma \beta \Gamma \nabla^H \big)\big|_{\Omega^{\bar k}_{+,(\lambda,\infty)}(M,E)}  \exp \big(-\frac{1}{2}u(\Gamma \nabla^H)^2 \big|_{\Omega^{\bar k}_{+,(\lambda,\infty)}(M,E)}\big)\nonumber \\ & \qquad \cdot \exp \big(-\frac{1}{2}u(\Gamma \nabla^H)^2 \big|_{\Omega^{\bar k}_{+,(\lambda,\infty)}(M,E)}\big) \cdot \big(\Gamma  \nabla^H \Gamma \big)\big|_{\Omega^{\overline{k+1}}_{-,(\lambda,\infty)}(M,E)} \Big]  \nonumber \\
& = \Tr \Big[ \big( \Gamma \beta \Gamma \big)\big|_{\Omega^{\overline{k+1}}_{-,(\lambda,\infty)}(M,E)} \cdot \big( \nabla^H \Gamma \big)^2\big|_{\Omega^{\overline{k+1}}_{-,(\lambda,\infty)}(M,E)} \exp \big(-u( \nabla^H \Gamma)^2 \big|_{\Omega^{\overline{k+1}}_{-,(\lambda,\infty)}(M,E)}\big) \Big].   
\end{align}
For the last equality of \eqref{E:fvh4}, we used the fact that
\begin{align*}
& \nabla^H\big|_{\Omega^{\bar k}_{+,(\lambda,\infty)}(M,E)} \exp \big(-u(\Gamma \nabla^H)^2 \big|_{\Omega^{\bar k}_{+,(\lambda,\infty)}(M,E)}\big)\big(\Gamma  \nabla^H \Gamma \big)\big|_{\Omega^{\overline{k+1}}_{-,(\lambda,\infty)}(M,E)}\\
& = \big( \nabla^H \Gamma \big)^2\big|_{\Omega^{\overline{k+1}}_{-,(\lambda,\infty)}(M,E)} \exp \big(-u( \nabla^H \Gamma)^2 \big|_{\Omega^{\overline{k+1}}_{-,(\lambda,\infty)}(M,E)}\big).
\end{align*}
By combining \eqref{E:fvh3} with \eqref{E:fvh4}, we obtain
\begin{align}\label{E:fvha1}
& \sum_{k=0,1}(-1)^k \int_0^\infty u^{s}  \Tr \Big[ \big(\Gamma [\beta, \nabla^H] \Gamma \nabla^H \big)\big|_{\Omega^{\bar k}_{+,(\lambda,\infty)}(M,E)}  \exp \big(-u(\Gamma \nabla^H)^2 \big|_{\Omega^{\bar k}_{+,(\lambda,\infty)}(M,E)}\big) \Big] du \nonumber \\
& = \sum_{k=0,1}(-1)^k \int_0^\infty u^{s} \Big( \, \Tr \Big[ \big(\Gamma \beta \Gamma \big)\big|_{\Omega^{\bar k}_{+,(\lambda,\infty)}(M,E)} \cdot \big( \Gamma \nabla^H \big)^2 \big|_{\Omega^{\bar k}_{+,(\lambda,\infty)}(M,E)}  \exp \big(-u(\Gamma \nabla^H)^2 \big|_{\Omega^{\bar k}_{+,(\lambda,\infty)}(M,E)}\big) \Big] \nonumber \\
& \qquad - \Tr \Big[ \big( \Gamma \beta \Gamma \big)\big|_{\Omega^{\overline{k+1}}_{-,(\lambda,\infty)}(M,E)} \cdot \big( \nabla^H \Gamma \big)^2\big|_{\Omega^{\overline{k+1}}_{-,(\lambda,\infty)}(M,E)} \exp \big(-u( \nabla^H \Gamma)^2 \big|_{\Omega^{\overline{k+1}}_{-,(\lambda,\infty)}(M,E)}\big) \Big] \, \Big) du \nonumber \\
 & = \sum_{k=0,1}(-1)^k \int_0^\infty u^{s} \Tr \Big[ \big(\Gamma \beta \Gamma \big)\big|_{\Omega^{\bar k}_{(\lambda,\infty)}(M,E)}\cdot (\B^H_{(\lambda,\infty)})^2\exp \big(-u(\B^H_{(\lambda,\infty)})^2  \Big]du.
\end{align}
Similarly, we have
\begin{align}\label{E:fvh5}
& \Tr \Big[ \big(\Gamma \nabla^H \Gamma \nabla^H \beta  \big)\big|_{\Omega^{\bar k}_{+,(\lambda,\infty)}(M,E)}  \exp \big(-u(\Gamma \nabla^H)^2 \big|_{\Omega^{\bar k}_{+,(\lambda,\infty)}(M,E)}\big) \Big] \nonumber  \\
& = \Tr \Big[  \beta\big|_{\Omega^{\bar k}_{+,(\lambda,\infty)}(M,E)}  \cdot \big( \Gamma \nabla^H \big)^2 \big|_{\Omega^{\bar k}_{+,(\lambda,\infty)}(M,E)}  \exp \big(-u(\Gamma \nabla^H)^2 \big|_{\Omega^{\bar k}_{+,(\lambda,\infty)}(M,E)}\big) \Big]
\end{align}
and
\begin{align}\label{E:fvh6}
& \Tr \Big[ \big(\Gamma  \nabla^H  \Gamma \beta \nabla^H \big)\big|_{\Omega^{\bar k}_{+,(\lambda,\infty)}(M,E)} \exp \big( -u(\Gamma \nabla^H)^2 \big|_{\Omega^{\bar k}_{+,(\lambda,\infty)}(M,E)} \big) \Big] \nonumber  \\
& = \Tr \Big[  \beta\big|_{\Omega^{\overline{k+1}}_{-,(\lambda,\infty)}(M,E)}  \cdot \big( \nabla^H \Gamma \big)^2\big|_{\Omega^{\overline{k+1}}_{-,(\lambda,\infty)}(M,E)} \exp \big(-u( \nabla^H \Gamma)^2 \big|_{\Omega^{\overline{k+1}}_{-,(\lambda,\infty)}(M,E)}\big) \Big]. 
\end{align}
By combining \eqref{E:fvh5} with \eqref{E:fvh6}, we have
\begin{align}\label{E:fvha2}
& \sum_{k=0,1}(-1)^k \int_0^\infty u^{s}  \Tr \Big[ \big(\Gamma \nabla^H \Gamma [\beta, \nabla^H]  \big)\big|_{\Omega^{\bar k}_{+,(\lambda,\infty)}(M,E)}  \exp \big(-u(\Gamma \nabla^H)^2 \big|_{\Omega^{\bar k}_{+,(\lambda,\infty)}(M,E)}\big) \Big] du \nonumber \\
& = \sum_{k=0,1}(-1)^k \int_0^\infty u^{s} \Big( \, \Tr \Big[  \beta\big|_{\Omega^{\bar k}_{+,(\lambda,\infty)}(M,E)}  \cdot \big( \Gamma \nabla^H \big)^2 \big|_{\Omega^{\bar k}_{+,(\lambda,\infty)}(M,E)}  \exp \big(-u(\Gamma \nabla^H)^2 \big|_{\Omega^{\bar k}_{+,(\lambda,\infty)}(M,E)}\big) \Big] \nonumber \\
& \qquad - \Tr \Big[  \beta\big|_{\Omega^{\overline{k+1}}_{-,(\lambda,\infty)}(M,E)} \cdot \big( \nabla^H \Gamma \big)^2\big|_{\Omega^{\overline{k+1}}_{-,(\lambda,\infty)}(M,E)} \exp \big(-u( \nabla^H \Gamma)^2 \big|_{\Omega^{\overline{k+1}}_{-,(\lambda,\infty)}(M,E)}\big) \Big] \, \Big) du \nonumber \\
&  = \sum_{k=0,1}(-1)^k \int_0^\infty u^{s} \Tr \Big[ \beta\big|_{\Omega^{\bar k}_{(\lambda,\infty)}(M,E)} \cdot (\B^H_{(\lambda,\infty)})^2\exp \big(-u(\B^H_{(\lambda,\infty)})^2  \Big]du.
\end{align}
Combining \eqref{E:fvha1}, \eqref{E:fvha2}, with \eqref{E:fvh2}, we obtain
\begin{align*}
\frac{d}{dv}f(s,v) & =  -\sum_{k=0,1}(-1)^k \int_0^\infty u^{s} \Tr \Big[ \big(\Gamma \beta \Gamma \big)\big|_{\Omega^{\bar k}_{(\lambda,\infty)}(M,E)}\cdot (\B^H_{(\lambda,\infty)})^2\exp \big(-u(\B^H_{(\lambda,\infty)})^2  \Big]du \nonumber \\
 &  - \sum_{k=0,1}(-1)^k \int_0^\infty u^{s} \Tr \Big[  \beta\big|_{\Omega^{\bar k}_{(\lambda,\infty)}(M,E)} \cdot (\B^H_{(\lambda,\infty)})^2 \exp \big(-u(\B^H_{(\lambda,\infty)})^2 \big) \Big]du \\
 &= - 2\sum_{k=0,1}(-1)^k \int_0^\infty u^{s} \Tr \Big[  \beta\big|_{\Omega^{\bar k}_{(\lambda,\infty)}(M,E)} \cdot (\B^H_{(\lambda,\infty)})^2 \exp \big(-u(\B^H_{(\lambda,\infty)})^2 \big) \Big]du,
\end{align*}
where for the latter equality we used the fact that $\Tr (\beta \big|_{\Omega^{\bar k}_{[0,\lambda]}(M,E)})=\Tr (\Gamma \beta \Gamma  \big|_{\Omega^{\bar k}_{[0,\lambda]}(M,E)})$.
The rest is similar to the proof of Lemma \ref{L:meind}. 
\end{proof}

\begin{lemma}\label{L:flva2}
Under the same assumptions, along any one parameter deformation of $H$ that fixes the cohomology class $[H]$, the element can be chosen so that
\[
\frac{d}{dv} \rho_{\Gamma_{[0,\lambda]}} = - \sum_{k=0,1}(-1)^k \Tr (\beta \big|_{\Omega^{\bar k}_{[0,\lambda]}(M,E)}) \rho_{\Gamma_{[0,\lambda]}},
\]
where we identify $\Det\big(H^\bullet(M,E,H)\big)$ along the deformation using \eqref{E:epsiso}.
\end{lemma}
\begin{proof}
In order to compare the elements $\rho_{\Gamma_{[0,\lambda]}} \in \Det\big(H^\bullet(M,E,H)\big)$ at different values of $v$. We choose a reference point, say $v=0$, and let $H^{(0)}$, $\rho^{(0)}_{\Gamma_{[0,\lambda]}}$ be the values of $H$, $\rho_{\Gamma_{[0,\lambda]}}$, respectively, at $v=0$. By \eqref{E:epsiso}, we have the isomorphism
\[
\Det(\varepsilon_B): \Det\big(H^\bullet(M,E,H^{(0)})\big) \to \Det\big(H^\bullet(M,E,H)\big).
\]
Since $\varepsilon_B=e^\beta$ on $\Omega^\bullet(M,E)$, we have, for $k=0,1$,
\[
\frac{d}{dv}(\Det(\varepsilon_B))^{-1}\rho_{\Gamma_{[0,\lambda]}}\, = \, - \sum_{k=0,1}(-1)^k \Tr (\beta \big|_{\Omega^{\bar k}_{[0,\lambda]}(M,E)})(\Det(\varepsilon_B))^{-1} \rho_{\Gamma_{[0,\lambda]}}.
\]  
The result follows. 
\end{proof} 

The argument of the following lemma is similar to the argument of Lemma 9.4 of \cite{BK3}. 
\begin{lemma}\label{L:flva0}
For any $v_1,v_2 \in (v_0-\delta, v_0 + \delta)$, we have
\[
\eta_\lambda(\nabla^H,v_1)-\eta_\lambda(\nabla^H,v_2) \equiv \eta(\B^{H}_{\bar{0}}(v_1)) - \eta(\B^{H}_{\bar{0}}(v_2)), \quad \operatorname{mod} \ \mathbb{Z}. 
\]
\end{lemma}

Again we need to study the dependence of $\eta(\B^{H}_{\bar{0}})$ on the parameter $v$. As Lemma \ref{L:rhopm}, this was essentially done in \cite{APS2} and \cite[P. 52]{Gi}.
\begin{lemma}\label{L:flva3}
The function $\eta(\B^{H}_{\bar{0}}(v))-\operatorname{rank}(E)\eta_{\trivial}(v)$ is, modulo $\mathbb{Z}$, independent of $v \in \mathbb{R}$.
\end{lemma}

Now we have the main theorem of this subsection.
\begin{theorem}
Let $M$ be an odd dimensional oriented closed Riemannian manifold. Let $(E,\nabla,h^E)$ be a flat complex vector bundle over $M$. Suppose that $H$ and $H'$ are closed differential forms on $M$ of odd degrees representing the same deRham cohomology class, and let $B$ be an even form so that $H'=H-dB$. Then the refined analytic torsion $\rho_{\operatorname{an}}(\nabla^{H'})=\Det (\varepsilon_B)(\rho_{\operatorname{an}}(\nabla^{H}))$.
\end{theorem}
\begin{proof}
Again we choose a reference point, say $v=0$, and let $H^{(0)}$, $\rho^{(0)}_{\Gamma_{[0,\lambda]}}$ be the values of $H$, $\rho_{\Gamma_{[0,\lambda]}}$, respectively, at $v=0$. By \eqref{E:epsiso}, we have the isomorphism
\[
\Det(\varepsilon_B): \Det\big(H^\bullet(M,E,H^{(0)})\big) \to \Det\big(H^\bullet(M,E,H)\big).
\]
Recall that $\varepsilon_B=e^\beta$ on $\Omega^\bullet(M,E)$. By combining Lemma \ref{L:flva1}, Lemma \ref{L:flva2} with Lemma \ref{L:flva3}, we conclude that $(\Det(\varepsilon_B))^{-1}\rho_{\operatorname{an}}(\nabla^{H})$ is, up to sign, invariant along the deformation. Since the function $(\Det(\varepsilon_B))^{-1}\rho_{\operatorname{an}}(\nabla^{H})$ is continuous and nonzero, the sign in the right hand side of the equality must be positive. This proves the statement.
\end{proof}

As pointed out by Mathai and Wu, \cite{MW}, that when $H$ is a $3$-form on $M$, the deformation of the Riemannian metric $g^M$ and that of the flux $H$ within its cohomology class can be interpreted as a deformation of generalized metrics on $M$ and the analytic torsion should be defined for generalized metrics so that the deformations of $g^M$ and $H$ are unified. Similarly, the the refined analytic torsion should be also defined for generalized metrics so that the deformations of $g^M$ and $H$ are unified. 

\section{A duality theorem for the twisted refined analytic torsion}\label{S:dttrat}


In this section we first review the concept of the dual of a complex and construct a natural isomorphism between the determinant lines of a $\mathbb{Z}_2$-graded complex and its dual. We then show that this isomorphism is compatible with the canonical isomorphism \eqref{E:isomorphism}. Finally we establish a relationship between the twisted refined analytic torsion corresponding to a flat connection and that of its dual. The contents of this section are $\mathbb{Z}_2$-graded analogues of Sections 3 and 10 of \cite{BK3}. Throughout this section, $\bf{k}$ is a field of characteristic zero endowed with an involutive automorphism $\tau:\bf{k} \to \bf{k}.$
The main examples are $\bf{k}=\mathbb{C}$ with $\tau$ being the complex conjugation and $\bf{k}=\mathbb{R}$ with $\tau$ being the identity map.

\subsection{The $\mathbb{Z}_2$-graded $\tau$-dual space}\label{SS:duakcplx}
If $V,W$ are $\bf{k}$-vector spaces, a map $f:V \to W$ is said to be $\tau-linear$ if 
\[
f(x_1v_1+x_2v_2)=\tau(x_1)v_1+\tau(x_2)v_2, \, \text{for any} \, v_1,v_2 \in V, x_1,x_2 \in \bf{k}.
\]
The linear space $V^*=V^{*_\tau}$ of all $\tau$-linear maps $V \to \bf{k}$ is called the $\tau-dual \ space \  to \ V$. There are natural $\tau$-linear isomorphisms, cf. \cite[Subsection 3.1]{BK3}, 
\begin{equation}\label{E:albev}
\alpha_V:\Det(V^*) \rightarrow \Det(V)^{-1}, \quad \beta_V:\Det(V) \rightarrow \Det(V^*)^{-1}.
\end{equation}
Then for any $v \in \Det(V)$, we have, cf. \cite[(3-8)]{BK3},
\begin{equation}\label{E:alinvb}
\big( \alpha_V^{-1}(v^{-1}) \big)^{-1} = (-1)^{\dim V} \beta_V(v),
\end{equation}
Let $V$ and $W$ be $\bf{k}$-vector spaces, then, for any $v \in \Det(V),w \in \Det(W)$, we have, cf. \cite[(3-9)]{BK3}, 
\begin{equation}\label{E:dualmu}
\big( \mu_{V,W}(v \otimes w)\big)^{-1}= \alpha_{V \oplus W} \circ \mu_{V^*,W^*}\big( \alpha_V^{-1}(v^{-1}) \otimes \alpha_W^{-1}(w^{-1}) \big).
\end{equation}

Let $T:V \to W$ be a $\bf{k}$-linear map. The $\tau-adjoint \ of \ T$ is the linear map
\[
T^*:W^* \rightarrow V^*
\]
such that 
\[
(T^*w^*)(v) = w^*(Tv), \ \text{for all} \ v \in V, w^* \in W^*. 
\]
If $T$ is bijective, then, for any nonzero $v \in \Det(V)$, we have, cf. \cite[(3-11)]{BK3},
\begin{equation}\label{E:tastal}
T^*\alpha_W^{-1}\big( (Tv)^{-1} \big) = \alpha_V^{-1}(v^{-1}).
\end{equation}

Let $V^0, V^1, \cdots, V^m$ be finite dimensional $\bf{k}$-vector space, where $m=2r-1$ is an odd integer. Denote by $V^{\bar 0}=\bigoplus_{i=0}^{r-1} V^{2i}$ and $V^{\bar 1}=\bigoplus_{i=0}^{r-1} V^{2i+1}$. Let $V^{\bullet} = V^{\bar 0} \oplus V^{\bar 1}$ be a finite dimensional $\mathbb{Z}_2$-graded $\bf{k}$-vector space. We define the  $\mathbb{Z}_2-graded \ (\tau-) \ dual \ space\ \widehat{V} = \widehat{V}^{\bar 0} \oplus \widehat{V}^{\bar 1}$ by
\[
\widehat{V}^{\bar k} := (V^{\overline{k+1}})^*, \quad k=0,1.
\]
Then \eqref{E:albev} induces a $\tau$-linear isomorphism
\begin{equation}\label{E:alvhatv}
\alpha_{V^\bullet}:\Det(V^\bullet) \rightarrow \Det(\widehat{V}^\bullet),
\end{equation}
defined by
\begin{equation}\label{E:allinve1}
\alpha_{V^\bullet}(v_{\bar 0} \otimes (v_{\bar 1})^{-1}) = (-1)^{\mathcal{M}(V^\bullet)} \cdot \alpha^{-1}_{V^{\bar 1}}(v_{\bar 1}^{-1}) \otimes \alpha_{V^{\bar 0}}(v_{\bar 0}),
\end{equation}
where $v_{\bar k} \in \Det(V^{\bar k}), k=0,1$ and, cf. \eqref{E:mvw123}, 
\[
\mathcal{M}(V^\bullet) = \mathcal{M}(V^\bullet, V^\bullet)= \dim V^{\bar 0} \cdot \dim V^{\bar 1}.
\]

\subsection{The dual complex of a $\mathbb{Z}_2$-graded complex}
Consider the $\mathbb{Z}_2$-graded complex \eqref{E:detline} of finite dimensional $\bf{k}$-vector spaces. The dual complex of the $\mathbb{Z}_2$-graded complex \eqref{E:detline} is the complex
\begin{equation}\label{E:detline1} 
 (\widehat{C}^\bullet,d^*) \  : \  \cdots \stackrel{d^*}{\longrightarrow} \ \widehat{C}^{\bar{0}} \ \stackrel{d^*}{\longrightarrow}\ \widehat{C}^{\bar{1}}\ \stackrel{d^*}{\longrightarrow} \ \widehat{C}^{\bar{0}} \ \stackrel{d^*}{\longrightarrow} \cdots,
\end{equation}
where $\widehat{C}^{\bar k}=(C^{\overline{k+1}})^*$ and $d^*$ is the $\tau$-adjoint of $d$. Then the cohomology $H^{\bar k}(d^*)$ of $\widehat{C}^\bullet$ is natural isomorphic to the $\tau$-dual space to $H^{\overline{k+1}}(d) \, (k=0,1)$. Hence by \eqref{E:alvhatv}, we obtain $\tau$-linear isomorphisms
\begin{equation}\label{E:alphah0}
\begin{array}{l}
\alpha_{C^\bullet}:  \Det(C^\bullet) \rightarrow \Det(\widehat{C}^\bullet),
\\
\alpha_{H^\bullet(d)}:  \Det(H^\bullet(d)) \rightarrow \Det(H^\bullet(d^*)).
\end{array}
\end{equation}

The following lemma is the $\mathbb{Z}_2$-graded analogue of Lemma 3.6 of \cite{BK3}. The proof is similar to the proof of Lemma 3.6 of \cite{BK3}. We skip the proof.

\begin{lemma}\label{L:comm22}
Let $(C^\bullet,d)$ be a $\mathbb{Z}_2$-graded complex of finite dimensional $\bf{k}$-vector spaces, defined as \eqref{E:detline}. Further, assume that the Euler characteristics $\chi(C^\bullet)=\chi(\widehat{C}^\bullet)=0$. Then the following diagramm commutes:
\begin{equation}\label{E:comm2}
\begin{CD}
\Det(C^\bullet)     & @>\phi_{C^\bullet} >> & \Det(H^\bullet(d))     \\
  @V \alpha_{C^\bullet} VV &        & @VV  \alpha_{H^\bullet(d)} V   \\
  \Det(\widehat{C}^\bullet)     & @> \phi_{\widehat{C}^\bullet} >> & \Det\big( H^\bullet(d^*) \big) 
\end{CD}
\end{equation}
where $\phi_{C^\bullet}$ and $\phi_{\widehat{C}^\bullet}$ are defined as in \eqref{E:phic}. 
\end{lemma}
\subsection{The refined torsion of the $\mathbb{Z}_2$-graded dual complex}
Suppose now that $\bf{k}$ is endowed with an involutive endomorphism $\tau$. Let $\widehat{C}^\bullet$ be the $\tau$-dual complex of $C$ and let $\alpha_{C^\bullet}: \Det(C^\bullet) \to \Det(\widehat{C}^\bullet)$ denote the $\tau$-isomorphism defined in \eqref{E:alphah0}. Let $\Gamma^*$ be the $\tau$-adjoint of $\Gamma$. Then $\Gamma^*$ is a chirality operator for the complex $\widehat{C}^\bullet$.

The following lemma is the $\mathbb{Z}_2$-graded analogue of Lemma 4.11 of \cite{BK3}.
\begin{lemma}\label{L:rhostarrho}
In the situation described above,
\begin{equation}\label{E:lega00}
\rho_{\Gamma^*} = \alpha_{H^\bullet(d)}(\rho_\Gamma).
\end{equation}
\end{lemma}
\begin{proof}
Fix $c_{\bar 0} \in \Det(C^{\bar 0})$ and set 
\begin{equation}\label{E:lega1}
\widehat{c}_{\bar 0} = \alpha^{-1}_{C^{\bar 0}}\big( (\Gamma c_{\bar 0})^{-1} \big) \in \Det(\widehat{C}^{\bar 0}).
\end{equation}
Then, by \eqref{E:tastal}, 
\begin{equation}\label{E:lega2}
\Gamma^* \widehat{c}_{\bar 0} = \alpha^{-1}_{C^{\bar 0}}(c_{\bar 0}^{-1}) \in \Det(\widehat{C}^{\bar 0}).
\end{equation}
Using \eqref{E:alinvb}, we obtain from \eqref{E:lega1} and \eqref{E:lega2}, that
\begin{equation}\label{E:lega3}
\beta_{C^{\bar 0}}(c_{\bar 0}) = (-1)^{\dim C^{\bar 0}} \cdot (\Gamma^* \widehat{c}_{\bar 0})^{-1}, \quad \beta_{C^{\bar 1}}(\Gamma c_{\bar 0}) = (-1)^{\dim C^{\bar 0}} \cdot \widehat{c}_{\bar 0}^{-1}. 
\end{equation}
Combining \eqref{E:allinve1}, \eqref{E:elec}, \eqref{E:lega1}, \eqref{E:lega2} and \eqref{E:lega3}, we have
\begin{equation}\label{E:lega4}
\alpha_{C^\bullet}(c_\Gamma) = (-1)^{\mathcal{M}(C^\bullet)+\dim C^{\bar 0}} \cdot c_{\Gamma^*}.
\end{equation}
By \eqref{E:elerho}, $\rho_\Gamma= \phi_{C^\bullet}(c_\Gamma)$. Therefore, from Lemma \ref{L:comm22}, we obtain
\begin{equation}\label{E:lega5}
\alpha_{H^\bullet(d)}(\rho_\Gamma) = \phi_{\widehat{C}^\bullet} \circ \alpha_{C^\bullet}(c_\Gamma) = (-1)^{\mathcal{M}(C^\bullet)+\dim C^{\bar 0}} \cdot \rho_{\Gamma^*}.
\end{equation}
By \eqref{E:mvw123} and the fact $\dim C^{\bar 0} = \dim C^{\bar 1}$, we get
\begin{equation}\label{E:lega6}
\mathcal{M}(C^\bullet)+\dim C^{\bar 0} = \dim C^{\bar 1} \cdot \dim C^{\bar 0} + \dim C^{\bar 0} = \dim C^{\bar 0} \cdot (\dim C^{\bar 0} + 1) \equiv 0, \quad \text{mod} \ 2. 
\end{equation}
Combining \eqref{E:lega5} with \eqref{E:lega6}, we obtain \eqref{E:lega00}.
\end{proof}

\subsection{The duality theorem}\label{SS:duaconne1}
Suppose that $M$ is a closed oriented manifold of odd dimension $m=2r-1$. Let $E \to M$ be a complex vector bundle over $M$ and let $\nabla$ be a flat connection on $E$. Fix a Hermitian metric $h^E$ on $E$. Denote by $\nabla'$ the connection on  $E$ dual to the connection $\nabla$, cf. \cite[Subsection 10.1]{BK3}. We denote by $E'$ the flat bundle $(E,\nabla')$, referring to $E'$ as the dual of the flat vector bundle $E$. Using the construction of Section \ref{SS:duakcplx}, with $\tau: \mathbb{C} \to \mathbb{C}$ be the complex conjugation, we have the canonical anti-linear isomorphism 
\begin{equation}\label{E:antiiso}
\alpha : \Det \big( H^\bullet(M,E,H) \big) \rightarrow \Det \big( H^\bullet(M,E',H) \big).
\end{equation}
  
The following theorem is the main result of this section and is the twisted analogue of Theorem 10.3 of \cite{BK3}.
\begin{theorem}\label{T:duality0}
Let $E \to M$ be a complex vector bundle over a closed oriented odd dimensional manifold $M$ and let $\nabla$ be a flat connection on $E$. Denote by $\nabla'$ the connection dual to $\nabla$ with respect to a Hermitian metric $h^E$ on $E$ and let $H$ be a odd-degree cloed form, other than one form, on $M$. Then 
\begin{equation}\label{E:aloh15}
\alpha(\rho_{\an}(\nabla^H)) = \rho_{\an}(\nabla'^H)\cdot e^{2 \pi i \big( \bar{\eta}(\nabla^H,g^M)-(\operatorname{rank}E)\eta_{\trivial}(g^M) \big)},
\end{equation}
where $\alpha$ is the anti-linear isomorphism \eqref{E:antiiso} and $g^M$ is any Riemannian metric on $M$.
\end{theorem}

The rest of this section is concerned about the proof of Thoerem \ref{T:duality0}.

\subsection{A choice of $\lambda$}\label{SS:anglecho}
Assume that no eigenvalue of $\B^H_{(\lambda,\infty)}$ lies in the solid angles $L_{[-\theta-\pi,\theta]}$ and $L_{[-\theta,\theta+\pi]}$, cf. \cite[Subsection 10.4 and 10.5]{BK3}, then it follows that no eigenvalue of $(\B^H_{(\lambda,\infty)})^2$ lies in the solid angles $L_{[-2\theta,2\theta+2\pi]}$. 

Let $\B'^H$ denote the twisted odd signature operator associated to the connection $\nabla'$, the odd-degree form $H$ and the Riemannian metric $g^M$. One can check that 
\begin{equation}\label{E:nahdu1}
(\nabla^H)^*=\Gamma \nabla'^H \Gamma \ \text{and} \ (\nabla'^H)^*=\Gamma \nabla^H \Gamma.
\end{equation}
Using \eqref{E:toso}, \eqref{E:nahdu1} and the equality $\Gamma^*=\Gamma$, one can see that the adjoint $\B'^H$ of $\B^H$ satisfies 
\begin{equation}\label{E:nahdu2}
(\B^H)^*=\B'^H.
\end{equation}
The choice of the angle $\theta$ guarantees that $\pm 2 \theta$ are Agmon angles for the operator $(\Gamma \nabla'^H)^2=\big((\Gamma \nabla^H)^2\big)^*$. In particular, for each $\lambda \ge 0$, the number $\xi_\lambda(\nabla'^H,g^M,\theta)$ can be defined by the formula \eqref{E:xino}, with the same angle $\theta$ and with $\nabla^H$ replaced by $\nabla'^H$ everywhere.

The following lemma is twisted analogue of Lemma 10.6 of \cite{BK3} and the proof is similar to the proof of Lemma 10.6 of \cite{BK3}. We skip the proof.
\begin{lemma}\label{L:xiducong}  
Let $\theta$ be as above and let $\lambda \ge 0$ be big enough so that the operator $\B^H_{(\lambda,\infty)}$ does not have purely imaginary eigenvalues, cf. \cite[Subsection 10.5]{BK3}. Then
\[
\xi_\lambda(\nabla'^H,g^M,\theta) = \bar{\xi}_\lambda(\nabla^H,g^M,\theta),
\]
and
\begin{equation}\label{E:etanooo}
\eta_\lambda(\nabla'^H)=\bar{\eta}_\lambda(\nabla^H),
\end{equation}
where $\bar{z}$ denotes the complex conjugate of the number $z \in \mathbb{C}$.
\end{lemma}

\subsection{Small eigenvalues of $\B^H$ and $\B'^H$}
We define $\Omega^{\bar k}_{\pm,\mathcal{I}}(M,E'), \Omega^{\bar k}_{\pm}(M,E')$, and $\Omega^{\bar k}(M,E')$ in similar ways as $\Omega^{\bar k}_{\pm,\mathcal{I}}(M,E), \Omega^{\bar k}_{\pm}(M,E)$ and $\Omega^{\bar k}(M,E)$, respectively. As \eqref{E:dno}, for $k=0,1$, set
\begin{equation}\label{E:dno1}
d_{\bar{k},\lambda}^{\pm} = \dim \Omega^{\bar k}_{\pm,[0,\lambda]}(M,E), \quad {d'}_{\bar{k},\lambda}^{\pm} = \dim \Omega^{\bar k}_{\pm,[0,\lambda]}(M,E').
\end{equation}
From the fact that
$\Gamma \B_{\bar{k},(\lambda, \infty)}^{H,\pm} \Gamma = \B_{\overline{k+1},(\lambda, \infty)}^{H,\mp}$, we conclude that
\begin{equation}\label{E:aaa}
\Gamma \Big( \Omega^{\bar k}_{\pm,[0,\lambda]}(M,E) \Big) = \Big( \Omega^{\overline{k+1}}_{\mp,[0,\lambda]}(M,E) \Big), \quad k=0,1.
\end{equation}
Therefore 
\[
d_{\bar{k},\lambda}^{\pm} = d_{\overline{k+1},\lambda}^{\mp}, \quad k=0,1.
\]
Hence
\begin{equation}\label{E:ddequ1}
\sum_{k=0,1}(-1)^k d_{\bar{k},\lambda}^- = d_{\bar{0},\lambda}^- - d_{\bar{1},\lambda}^- \equiv d_{\bar{0},\lambda}^- + d_{\bar{0},\lambda}^+ = \dim \Omega^{\bar 0}_{[0,\lambda]}(M,E), \quad \text{mod} \ 2\mathbb{Z}.
\end{equation}
From \eqref{E:nahdu2} and \eqref{E:aaa}, we obtain
\begin{equation}\label{E:ddequ2}
\dim \Omega^{\bar 0}_{[0,\lambda]}(M,E)=\dim \Omega^{\bar 1}_{[0,\lambda]}(M,E)=\dim \Omega^{\bar 0}_{[0,\lambda]}(M,E')=\dim \Omega^{\bar 1}_{[0,\lambda]}(M,E').
\end{equation}
Hence by \eqref{E:ddequ1} and \eqref{E:ddequ2}, we have
\begin{equation}\label{E:ddequ3}
\sum_{k=0,1}(-1)^k {d'}_{\bar{k},\lambda}^- = \dim \Omega^{\bar 0}_{[0,\lambda]}(M,E), \quad \text{mod} \ 2\mathbb{Z}.
\end{equation}

By the definition \eqref{E:etainv11}, we obtain 
\begin{equation}\label{E:ddequ4}
2\eta(\B^H_{\bar{0},[0,\lambda]}) \equiv \dim \Omega^{\bar 0}_{[0,\lambda]}(M,E), \quad \text{mod} \ 2\mathbb{Z}.
\end{equation}

From \eqref{E:etano}, \eqref{E:ddequ3} and \eqref{E:ddequ4}, we obtain, modulo $2\mathbb{Z}$,
\begin{align}\label{E:ddequ5}
2\eta(\B^H_{\bar{0}}(\nabla^H)) & =  2 \eta(\B^H_{\bar{0},(\lambda, \infty)}(\nabla^H)) + 2\eta(\B^H_{\bar{0},[0,\lambda]}) \nonumber \\
& \equiv  2 \eta_\lambda(\nabla^H) + \sum_{k=0,1}d^-_{\bar{k},\lambda}.
\end{align}
Similarly,
\begin{equation}\label{E:ddequ6}
2\eta(\B^H_{\bar{0}}(\nabla'^H)) \equiv 2 \eta_\lambda(\nabla'^H) + \sum_{k=0,1}d^-_{\bar{k},\lambda}, \quad \text{mod} \ 2\mathbb{Z}.
\end{equation}

\subsection{Proof of Theorem \ref{T:duality0}}
Let $\rho'_{\Gamma_{[0,\lambda]}}$ be the twisted refined torsion of the complex $\Omega_{[0,\lambda]}^\bullet(M,E')$ associated to the restriction of $\Gamma$ to $\Omega_{[0,\lambda]}^\bullet(M,E')$.

By Lemma \ref{L:rhostarrho}, \eqref{E:antiiso} and the fact that $\Gamma^*=\Gamma$, cf. \cite[Proposition 3.58]{BGV},
we obtain 
\begin{equation}\label{E:rhpri5}
\rho'_{\Gamma_{[0,\lambda]}} = \alpha(\rho_{\Gamma_{[0,\lambda]}})
\end{equation}
From \eqref{E:rhoh}, \eqref{E:ldetetarell} and Definition \ref{D:maindef}, we obtain 
\begin{equation}\label{E:ldetetarell123}
\rho_{\operatorname{an}}(\nabla^H) = \rho_{\Gamma_{[0,\lambda]}} \cdot \exp\Big(\xi_\lambda(\nabla^H,g^M,\theta) - i \pi \eta_\lambda(\nabla^H) -\frac{i \pi}{2}\sum_{k=0,1}(-1)^k d^-_{\bar{k},\lambda} + i \pi (\operatorname{rank} (E))\eta_{\trivial} \Big).
\end{equation} 
Since $\alpha$ is an anti-linear isomorphism, $\alpha(\rho_{\an} \cdot z)=\alpha(\rho_{\an}) \cdot \bar{z}$ for any $z \in \mathbb{C}$. Hence, from \eqref{E:rhpri5} and \eqref{E:ldetetarell123}, we get
\begin{equation}\label{E:rhpri3}
\alpha \big( \rho_{\operatorname{an}}(\nabla^H) \big) = \rho'_{\Gamma_{[0,\lambda]}} \cdot \exp\Big(\bar{\xi}_\lambda(\nabla^H,g^M,\theta) + i \pi \bar{\eta}_\lambda(\nabla^H) + \frac{i \pi}{2}\sum_{k=0,1}(-1)^k d^-_{\bar{k},\lambda} - i \pi (\operatorname{rank} (E))\eta_{\trivial} \Big).
\end{equation}
Using Lemma \ref{L:xiducong} and the analogue of \eqref{E:ldetetarell123} for $\rho_{\an}(\nabla'^H)$, we obtain from \eqref{E:rhpri3}
\begin{equation}\label{E:rhpri30}
\alpha \big( \rho_{\operatorname{an}}(\nabla^H) \big) = \rho_{\operatorname{an}}(\nabla'^H) \cdot \exp\Big(2 i \pi \bar{\eta}_\lambda(\nabla^H) + i \pi \sum_{k=0,1}(-1)^k d^-_{\bar{k},\lambda} - i \pi (\operatorname{rank} (E))\eta_{\trivial} \Big).
\end{equation}
From \eqref{E:rhpri30} and \eqref{E:ddequ6}, we obtain \eqref{E:aloh15}.


\section{Comparison with the twisted analytic torsion}\label{S:cwtat}
In this section we first define the twisted Ray-Singer metric $\| \cdot \|^{\RS}_{\Det ( H^\bullet(M,E,H) )}$ and then calculate the twisted Ray-Singer norm $\| \rho_{\an}(\nabla^H) \|^{\RS}_{\Det ( H^\bullet(M,E,H) )}$ of the twisted refined analytic torsion. In particular, we show that, if $\nabla$ is a Hermitian metric, then $\| \rho_{\an}(\nabla^H) \|^{\RS}_{\Det ( H^\bullet(M,E,H) )}=1$.

\subsection{The twisted analytic torsion}\label{SS:twrst1}
Let $E \to M$ be a complex vector bundle over a cloed oriented manifold $M$ of odd dimension $m=2r-1$. Let $\nabla$ be a flat connection on $E$ and $H$ be a odd degree closed form, other than one form, on $M$. Fix a Riemannian metric $g^M$ on $M$ and a Hermitian metric $h^E$ on $E$. Let ${\nabla^H}^*$ denote the adjoint of $\nabla^H:=\nabla +H \wedge \cdot$ with respect to the scalar product $<\cdot, \cdot>_M$ on $\Omega^\bullet(M,E)$ defined by $h^E$ and the Riemannian metric $g^M$.


Now let
\[
\Delta^H = {\nabla^H}^* \nabla^H + \nabla^H {\nabla^H}^*
\]
be the Laplacian twisted by the form $H$. We denote by $\Delta^H_{\bar k}$ the restriction of $\Delta^H$ to $\Omega^{\bar k}(M,E), k=0,1$. Assume that $\mathcal{I}$ is an interval of the form $[0,\lambda], (\lambda, \mu], (\lambda, \infty) (\mu \ge \lambda \ge 0)$ and let $\Pi_{\Delta^H_{\bar{k}},\mathcal{I}}$ be the spectral projection of $\Delta^H_{\bar k}$ corresponding to $\mathcal{I}$, cf. Subsection~\ref{SS:gdto}. Set
\[
\check{\Omega}^{\bar k}_{\mathcal{I}}(M,E):= \Pi_{\Delta^H_{\bar k},\mathcal{I}} \big( \Omega^\bullet(M,E) \big) \subset \Omega^{\bullet}(M,E).
\]
Let $\Delta^{H,\mathcal{I}}_k$ denote the restriction of $\Delta^H_{\bar k}$ to $\check{\Omega}^{\bar k}_{\mathcal{I}}(M,E)$ and define
\begin{equation}\label{E:trs11}
\trs_{\mathcal{I}} = \trs_{\mathcal{I}}(\nabla^H):=\exp \Big( \frac{1}{2} \sum_{k=0,1}(-1)^{k+1} \LDet'_{-\pi}\big( {\nabla^H}^* \nabla^H  \big)\big|_{\Omega^{\bar k}_{\mathcal{I}}(M,E)} \Big).
\end{equation}
It is not difficult to check that, for any non-negative, real numbers $\mu \ge \lambda \ge 0$,
\begin{equation}\label{E:spli2}
\trs_{(\lambda,\infty)} = \trs_{(\lambda,\mu]}\cdot \trs_{[\mu,\infty)}.
\end{equation}

Note that if $\eta_{\bar k}$ is the unit volume element of $H^{\bar k}(M,E,H), k=0,1$, then 
\begin{equation}\label{E:tatmw}
\tau(M,E,H):= \big(\trs_{(0,\infty)}\big)^{-1} \cdot \eta_{\bar 0} \otimes \eta_{\bar 1}^{-1} \in \Det \big( H^\bullet(M,E,H) \big)
\end{equation} 
is the {\em twisted analytic torsion}, introduced by V. Mathai and S. Wu in \cite{MW}.

For each $\lambda >0$, the cohomology of the finite dimensional complex $\big( \check{\Omega}^\bullet_{[0,\lambda]}(M,E),\nabla^H \big)$ is naturally isomorphic to $H^\bullet(M,E,H)$. Identifying these two cohomology spaces, we then obtain from \eqref{E:phic} an isomorphism
\begin{equation}\label{E:philam}
\phi_\lambda = \phi_{\check{\Omega}^\bullet_{[0,\lambda]}(M,E)} : \Det \big( \check{\Omega}^\bullet_{[0,\lambda]}(M,E) \big) \rightarrow \Det \big( H^\bullet(M,E,H) \big).
\end{equation}
The scalar product $<\cdot,\cdot>$ on $\check{\Omega}^\bullet_{[0,\lambda]}(M,E) \subset \Omega^\bullet(M,E)$ defined by $g^M$ and $h^E$ induces a metric ${\| \cdot \|}_{\Det \big( \check{\Omega}^\bullet_{[0,\lambda]}(M,E)\big)}$ on the determinant line $\Det \big( \check{\Omega}^\bullet_{[0,\lambda]}(M,E)\big)$. Let ${\|\cdot \|_{\lambda}}$ denote the metric on the determinant line $\Det\big(H^\bullet(M,E,H) \big)$ such that the isomorphism \eqref{E:philam} is an isometry. Then, for $c \in \Det \big( \check{\Omega}^\bullet_{[0,\lambda]}(M,E) \big)$, we have
\begin{equation}\label{E:isoequiv}
{\| c \|}_{\Det \big( \check{\Omega}^\bullet_{[0,\lambda]}(M,E)\big)} = {\|\phi_\lambda(c) \|_{\lambda}}.
\end{equation}

Using the Hodge theory, we have the canonical identification 
\[
H^{\bar k}(M,E,H) \cong \Ker \Delta^H_{\bar k}, \quad k=0,1.
\]
By their inclusion in $\Omega^{\bar k}(M,E)$, the space of twisted harmonic forms $\Ker \Delta^H_{\bar k}$ inherits a metric. We denote by $|\cdot|_{\Det\big( H^\bullet(M,E,H) \big)}$ the corresponding metric on $\Det\big( H^\bullet(M,E,H) \big)$. By definition 
\begin{equation}\label{E:metricequ}
{\| \cdot \|}_{\Det \big( \check{\Omega}^\bullet_{\{0\}}(M,E)\big)} = |\cdot|_{\Det\big( H^\bullet(M,E,H) \big)}.
\end{equation}

The following is twisted analogue of \cite[Proposition 1.5]{BGS}. See also \cite{RuSe} for the contact version.
\begin{proposition}
\begin{equation}\label{E:metrican}
\|\cdot \|_\lambda = |\cdot|_{\Det\big( H^\bullet(M,E,H) \big)} \cdot \trs_{(0,\lambda]}.
\end{equation}
\end{proposition}
\begin{proof}
For $k=0,1$, fix $c'_{\bar k} \in  \Det \big( \check{\Omega}_{\{0\}}^{\bar k}(M,E)\big) \cong \Det \big( H^{\bar k}(M,E,H) \big)$ and $c''_{\bar k} \in \Det \big( \check{\Omega}_{(0,\lambda]}^{\bar k}(M,E)\big)$. Then, using the natural isomorphism
\begin{align*}
& \Det \big( \check{\Omega}_{\{0\}}^{\bar k}(M,E)\big) \otimes \Det \big( \check{\Omega}_{(0,\lambda]}^{\bar k}(M,E)\big) \nonumber \\  
& \cong  \Det \big(\check{\Omega}_{\{0\}}^{\bar k}(M,E) \oplus \check{\Omega}_{(0,\lambda]}^{\bar k}(M,E) \big) \nonumber \\
& =  \Det \big( \check{\Omega}_{[0,\lambda]}^{\bar k}(M,E) \big),
\end{align*}
we can regard the tensor product $c_{\bar k}:= c'_{\bar k} \otimes c''_{\bar k}$ as an element of $\Det \big( \check{\Omega}_{[0,\lambda]}^{\bar k}(M,E) \big)$. Denote by 
\[
\check{\Omega}^{\bar{k},-}_{\mathcal{I}}(M,E) = \Ker \nabla^H \cap \check{\Omega}^{\bar{k}}_{\mathcal{I}}(M,E), \, \quad
\check{\Omega}^{\bar{k},+}_{\mathcal{I}}(M,E) = \Ker {\nabla^H}^* \cap \check{\Omega}^{\bar{k}}_{\mathcal{I}}(M,E).
\]
Since the complexes $\check{\Omega}^{\bar k}_{(0,\lambda]}(M,E), k=0,1$, are acyclic, it is not difficult to see that each $\check{\Omega}^{\bar k}_{(0,\lambda]}(M,E)$ splits orthogonally into 
\begin{equation}\label{E:splittt}
\check{\Omega}^{\bar k}_{(0,\lambda]}(M,E) = \check{\Omega}^{\bar{k},-}_{(0,\lambda]}(M,E) \oplus \check{\Omega}^{\bar{k},+}_{(0,\lambda]}(M,E).
\end{equation}
Take $a_{\bar k} \in \Det \big(\check{\Omega}^{\bar{k},+}_{(0,\lambda]}(M,E)\big)$ so that $c''_{\bar k}=\nabla^H (a_{\overline{k+1}}) \wedge a^{\bar k}$. Denote by $c=c'\otimes c''$, where $c'=c'_{\bar 0} \otimes {c'_{\bar 1}}^{-1}$ and $c''=c''_{\bar 0} \otimes {c''}^{-1}_{\bar 1}$. Then 
\begin{align}\label{E:split1}
\|c \|_{\Det \big( \check{\Omega}^{\bullet}_{[0,\lambda]}(M,E) \big)} & =  \|c'\|_{\Det\big( \check{\Omega}_{\{0\}}^\bullet(M,E) \big)} \times \| c'' \|_{\Det \big( \check{\Omega}_{(0,\lambda]}^{\bar k}(M,E)\big)}\nonumber \\ 
& = \|c'\|_{\Det\big( \check{\Omega}_{\{0\}}^\bullet(M,E) \big)} \times \|a_{\bar 0}\|_{\Det \big( \check{\Omega}^{\bar{0},+}_{(0,\lambda]}(M,E) \big)}\times \|\nabla^H(a_{\bar 1})\|_{\Det \big( \check{\Omega}^{\bar{0},-}_{(0,\lambda]}(M,E) \big)}
\nonumber  \\ 
& \times \Big(\|a_{\bar 1}\|_{\Det \big( \check{\Omega}^{\bar{1},+}_{(0,\lambda]}(M,E) \big)} \Big)^{-1}
  \times \Big(\|\nabla^H(a_{\bar 0})\|_{\Det \big( \check{\Omega}^{\bar{1},-}_{(0,\lambda]}(M,E) \big)} \Big)^{-1}.
\end{align}
where $\| \cdot \|_{V}$ denotes the naturally induced norm on the subspace $V$. The space $\check{\Omega}^{\bar k}_{(0,\lambda]}(M,E)$ splits orthogonally into eigenspaces.
\[
\check{\Omega}^{\bar k}_{(0,\lambda]}(M,E)= \oplus_{\nu \le \lambda}\check{\Omega}^{\bar k}_{\{\nu\}}(M,E)
\]
Given $\nu \in (0,\lambda]$, we choose an orthogonal basis $(v_1, \cdots, v_{n_k})$ of each eigenspace $\check{\Omega}^{\bar{k},+}_{\{ \nu \}}(M,E)$ and choose the element $a_{\bar{k},\nu}=v_1 \wedge \cdots \wedge v_{n_k} \in \Det \big( H^{\bar{k},+}_{\{\nu \}}(M,E,H) \big)$, where $n_k=\dim \check{\Omega}^{\bar{k},+}_{\{ \nu \}}(M,E)$. Then
\begin{align}\label{E:split2}
\| \nabla^H (a_{\bar{k},\nu}) \|_{\Det\big( \check{\Omega}^{\overline{k+1},-}_{\{ \nu \}}(M,E) \big)} & =  \| \nabla^H v_1 \wedge \cdots \wedge \nabla^H v_{n_k}\|_{\Det\big( \check{\Omega}^{\overline{k+1},-}_{\{ \nu \}}(M,E) \big)} \nonumber \\
& =  \| \nabla^H v_1\|_{\Det\big( \check{\Omega}^{\overline{k+1},-}_{\{ \nu \}}(M,E) \big)} \times  \cdots \times \| \nabla^H v_{n_k}\|_{\Det\big( \check{\Omega}^{\overline{k+1},-}_{\{ \nu \}}(M,E) \big)} \nonumber \\ 
& =   \nu^{\frac{n_k}{2}} \|v_1\|_{\Det\big( \check{\Omega}^{\bar{k},+}_{\{ \nu \}}(M,E) \big)} \times \cdots \times \|v_{n_k}\|_{\Det\big( \check{\Omega}^{\bar{k},+}_{\{ \nu \}}(M,E) \big)} \nonumber \\
 & = \nu^{\frac{n_k}{2}}\| a_{\bar{k},\nu}\|_{\Det\big( \check{\Omega}^{\bar{k},+}_{\{ \nu \}}(M,E) \big)}. 
\end{align}
By combining \eqref{E:split1} with \eqref{E:split2}, we obtain
\begin{equation}\label{E:split3}
{\|c \|}_{\Det\big( \check{\Omega}_{[0,\lambda]}^\bullet(M,E) \big)} = \|c'\|_{\Det\big( \check{\Omega}_{\{0\}}^\bullet(M,E) \big)} \times \prod_{k=0,1} \big(\Det_{-\pi}\big( \nabla^H {\nabla^H}^* \big) \big|_{ \Omega^{\bar k}_{(0,\lambda]}(M,E)} \big)^{(-1)^{k+1}/2}.
\end{equation}
By \eqref{E:trs11}, \eqref{E:isoequiv}, \eqref{E:metricequ} and \eqref{E:split3}, we obtain the result.
\end{proof}
By \eqref{E:metrican}, we have, for $0 \le \lambda \le \mu$,
\begin{equation}\label{E:spli1}
\|\cdot \|_\mu = \| \cdot \|_\lambda \cdot \trs_{(\lambda,\mu]}.
\end{equation} 
The {\em twisted Ray-Singer metric} on $\Det \big( H^\bullet(M,E,H) \big)$ is defined by the formula
\begin{equation}\label{E:metric}
\| \cdot \|^{\RS}_{\Det \big( H^\bullet(M,E,H) \big)}:= \| \cdot \|_\lambda \cdot \trs_{(\lambda,\infty)}, \quad \lambda \ge 0.
\end{equation}
It follows immediately from \eqref{E:spli2} and \eqref{E:spli1} that $\| \cdot \|_{\Det \big( H^\bullet(M,E,H) \big)}$ is independent of the choice of $\lambda \ge 0$. Note that for $\lambda =0$, by \eqref{E:tatmw} and \eqref{E:metric}, we have
\[
\| \tau(M,E,H) \|^{\RS}_{\Det \big( H^\bullet(M,E,H) \big)}=1.
\]

\begin{theorem}\label{T:comthm1}
Let $E$ be a complex vector bundle over a closed oriented odd dimensional manifold $M$ and let $\nabla$ be a flat connection on $E$. Further, let $H$ be a odd-degree closed form on $M$ and denote by $\nabla^H:=\nabla+H \wedge \cdot$. Then
\begin{equation}\label{E:comthm2}
\| \rho_{\an}(\nabla^H) \|^{\RS}_{\Det\big( H^\bullet(M,E,H) \big)}= e^{\pi \operatorname{Im} \eta(\nabla^H,g^M)},
\end{equation}
where 
\[
\eta(\nabla^H,g^M) = \eta\big( \B_{\bar 0}(\nabla^H,g^M) \big).
\]
In particular, if $\nabla$ is a Hermitian connection, then $\| \rho_{\an}(\nabla^H) \|^{\RS}_{\Det\big( H^\bullet(M,E,H) \big)}=1$.
\end{theorem}

The rest of this section is concerned with the proof of Theorem \ref{T:comthm1}.

\subsection{Comparison between the twisted Ray-Singer metrics associated to a connection and to its dual}
We assume that $\theta \in (-\pi/2,0)$ is any Agmon angle for the twisted odd signature operator $\B^H$ such that no eigenvalue of $\B^H_{(\lambda,\infty)}$ lies in the solid angles $L_{[-\theta-\pi,-\pi/2]}$, $L_{(-\pi/2,\theta]}$, $L_{[-\theta,\pi/2)}$ and $L_{(\pi/2,\theta+\pi]}$, cf. \cite[Subsection 11.4]{BK3}

As in Subsection \ref{SS:duaconne1}, let $\nabla'$ be the connection dual to $\nabla$ with respect to the Hermitian metric $h^E$ and let $E'$ denote the flat bundle $(E,\nabla')$. Let $H$ be an odd degree closed form, other than one form, on $M$ and denote by $\nabla^H:=\nabla + H \wedge \cdot$. Let
\[
{\Delta'}^H=({\nabla'^H})^* \nabla'^H + \nabla'^H (\nabla'^H)^*,
\]
denote the twisted Laplacian of the connection $\nabla'$ twisted by the form $H$. For any $\lambda \ge 0$, we denote by
\[
\check{\Omega}^\bullet_{[0,\lambda]}(M,E') \subset \Omega^\bullet(M,E')
\]
the image of the spectral projection $\Pi_{\Delta'^H,[0,\lambda]}$, cf. Subsection \ref{SS:gdto}. As in Subsection \ref{SS:twrst1}, we use the scalar product induced by $g^M$ and $h^E$ on $\check{\Omega}^\bullet_{[0,\lambda]}(M,E')$ to construct a metric $\| \cdot \|'_\lambda$ on $\Det \big( H^\bullet(M,E',H) \big)$ and we define the twisted Ray-Singer metric on $\Det \big( H^\bullet(M,E',H) \big)$ by the formula
\begin{equation}\label{E:metricdual}
\| \cdot \|^{\RS}_{\Det \big( H^\bullet(M,E',H) \big)}:= \| \cdot \|'_\lambda \cdot \trs_{(\lambda,\infty)}(\nabla'^H), \quad \lambda \ge 0.
\end{equation}
As the untwisted case, cf. \cite[Subsection 11.6]{BK3}, we have the following identification,
\[
\check{\Omega}^\bullet_{[0,\lambda]}(M,E') \cong \check{\Omega}^{m-\bullet}_{[0,\lambda]}(M,E)^*,
\]
which preserves the scalar products induced by $g^M$ and $h^E$ on $\check{\Omega}^\bullet_{[0,\lambda]}(M,E')$ and $\check{\Omega}^{m-\bullet}_{[0,\lambda]}(M,E)^*$. Hence, the anti-linear isomorphism $\alpha$, cf. \eqref{E:antiiso}, is an isometry with respect to the metrics $\|\cdot \|_\lambda$ and $\|\cdot \|'_\lambda$. In particular,
\[
\|\rho_{\an}(\nabla^H) \|_\lambda = \|\alpha(\rho_{\an}(\nabla^H)) \|'_\lambda.
\] 
It follows from \eqref{E:aloh15} that 
\begin{equation}\label{E:rholamim}
\|\rho_{\an}(\nabla^H) \|_\lambda = \| \rho_{\an}({\nabla}'^H) \|'_\lambda \cdot e^{2\pi \operatorname{Im} \eta(\nabla^H,g^M)}.
\end{equation}

We need the following lemma. For untwisted case, cf. for example \cite[Lemma 8.8]{BK1}.
\begin{lemma}
\begin{equation}\label{E:trl32}
\trs_{(\lambda,\infty)}(\nabla'^H)=\trs_{(\lambda,\infty)}(\nabla^H).
\end{equation}
\end{lemma}
\begin{proof}
From \eqref{E:nahdu1}, we have
\begin{equation}\label{E:nahnainterw}
(\nabla^H)^*\nabla^H = \Gamma \nabla'^H \Gamma \nabla^H = \Gamma (\nabla'^H \Gamma \nabla^H \Gamma) \Gamma = \Gamma \nabla'^H (\nabla'^H)^* \Gamma.
\end{equation}
As in \eqref{E:splittt}, we have
\begin{equation}\label{E:splittt0}
\check{\Omega}^{\bar k}_{(\lambda,\infty)}(M,E) = \check{\Omega}^{\bar{k},-}_{(\lambda,\infty)}(M,E) \oplus \check{\Omega}^{\bar{k},+}_{(\lambda,\infty)}(M,E).
\end{equation}
The operator $\nabla^H$ maps $\check{\Omega}^{\bar{k},+}_{(\lambda,\infty)}(M,E)$ isomorphically onto $\check{\Omega}^{\overline{k+1},-}_{(\lambda,\infty)}(M,E)$ and
\begin{equation}\label{E:intewin2} 
\nabla^H\big|_{\check{\Omega}^{\bar{k},+}_{(\lambda,\infty)}(M,E)} \big((\nabla^H)^* \nabla^H\big) \big|_{\check{\Omega}^{\bar{k},+}_{(\lambda,\infty)}(M,E)} =  \big(\nabla^H (\nabla^H)^* \big|_{\check{\Omega}^{\overline{k+1},-}_{(\lambda,\infty)}(M,E)}\big) \cdot \nabla^H\big|_{\check{\Omega}^{\bar{k},+}_{(\lambda,\infty)}(M,E)}.
\end{equation}
Hence, by \eqref{E:intewin2}, we have 
\begin{equation}\label{E:interwi1}
\LDet'_{-\pi}\big(  (\nabla'^H)^* \nabla'^H   \big)\big|_{\Omega^{\bar{k},+}_{(\lambda,\infty)}(M,E)} =  \LDet'_{-\pi}\big(   \nabla'^H (\nabla'^H)^*  \big)\big|_{\Omega^{\overline{k+1},-}_{(\lambda,\infty)}(M,E)}.
\end{equation}
From \eqref{E:nahnainterw}, we obtain
\begin{align*}
&\log \trs_{(\lambda,\infty)}(\nabla^H)\\
 &= \frac{1}{2} \sum_{k=0,1}(-1)^{k+1} \LDet'_{-\pi}\big( (\nabla^H)^* \nabla^H  \big)\big|_{\Omega^{\bar{k},+}_{(\lambda,\infty)}(M,E)} \\
& = \frac{1}{2} \sum_{k=0,1}(-1)^{k+1} \LDet'_{-\pi}\big( \Gamma \nabla'^H (\nabla'^H)^* \Gamma \big)\big|_{\Omega^{\bar{k},+}_{(\lambda,\infty)}(M,E)} \\
& = \frac{1}{2} \sum_{k=0,1}(-1)^{k+1} \LDet'_{-\pi}\big(  \nabla'^H (\nabla'^H)^*  \big)\big|_{\Omega^{\overline{m-k},-}_{(\lambda,\infty)}(M,E)} \\
 &= \frac{1}{2} \sum_{k=0,1}(-1)^k \LDet'_{-\pi}\big(  \nabla'^H (\nabla'^H)^*  \big)\big|_{\Omega^{\bar{k},-}_{(\lambda,\infty)}(M,E)} \\
&= \frac{1}{2} \sum_{k=0,1}(-1)^{k+1} \LDet'_{-\pi}\big(  (\nabla'^H)^* \nabla'^H   \big)\big|_{\Omega^{\bar{k},+}_{(\lambda,\infty)}(M,E)},
\end{align*}
where we use \eqref{E:interwi1} for the last equality. This proves the lemma.
\end{proof}

Hence, from \eqref{E:metric}, \eqref{E:metricdual}, \eqref{E:rholamim} and \eqref{E:trl32}, we conclude that
\begin{equation}\label{E:rhoimet}
\|\rho_{\an}(\nabla^H) \|^{\RS}_{\Det \big( H^\bullet(M,E,H) \big)} = \| \rho_{\an}({\nabla}'^H) \|^{\RS}_{\Det \big( H^\bullet(M,E',H) \big)} \cdot e^{2\pi \operatorname{Im} \eta(\nabla^H,g^M)}.
\end{equation}

\subsection{Direct sum of a connection and its dual}
Let 
\[
\widetilde{\nabla} =  \left(
\begin{array}{cc}
   \nabla & 0  \\
   0 & \nabla' 
\end{array}
\right)
\]
denote the flat connection on $E \oplus E$ obtained as a direct sum of the connections $\nabla$ and $\nabla'$. Denote by $$\widetilde{\nabla}^H:= \left(
\begin{array}{cc}
   \nabla^H & 0  \\
   0 & \nabla'^H 
\end{array}
\right).$$ 
Then a discussion similar to \cite[Subsection 11.7]{BK3}, where the untwisted case was treated, one easily obtains that, 
\[
\rho_{\an}(\widetilde{\nabla}^H)= \mu_{H^\bullet(M,E,H),H^\bullet(M,E',H)}\big( \rho_{\an}(\nabla^H) \otimes \rho_{\an}(\nabla'^H) \big)
\]
and
\[
\| \rho_{\an}(\widetilde{\nabla}^H) \|^{\RS}_{\Det \big( H^\bullet(M,E \oplus E',H \big)}=\| \rho_{\an}(\nabla^H) \|^{\RS}_{\Det \big( H^\bullet(M,E,H \big)} \cdot \| \rho_{\an}(\nabla'^H) \|^{\RS}_{\Det \big( H^\bullet(M, E',H \big)}.
\]
Combining this later equality with \eqref{E:rhoimet}, we get
\[
\| \rho_{\an}(\widetilde{\nabla}^H) \|^{\RS}_{\Det \big( H^\bullet(M,E \oplus E',H \big)} = \big( \| \rho_{\an}(\nabla^H) \|^{\RS}_{\Det \big( H^\bullet(M,E,H \big)} \big)^2 \cdot e^{-2\pi \operatorname{Im} \eta(\nabla^H,g^M)}.
\]
Hence, \eqref{E:comthm2} is equivalent to the equality
\begin{equation}\label{E:comrart}
\| \rho_{\an}(\widetilde{\nabla}^H) \|^{\RS}_{\Det \big( H^\bullet(M,E \oplus E',H \big)} = 1.
\end{equation}
By a slight modification of the deformation argument in \cite[Section 11, P.205-211]{BK3}, where the untwisted case was treated, we can obtain \eqref{E:comrart}. Hence, we finish the proof of Theorem \ref{T:comthm1}.

\end{document}